\documentclass[12pt]{article}

\usepackage{tikz}
\usetikzlibrary{intersections, calc, angles, arrows.meta}
\usetikzlibrary{patterns}

\usepackage{mathrsfs}
\usetikzlibrary{arrows}
\usepackage{subcaption}
\usepackage{indentfirst}
\usepackage{enumerate}
\usepackage{cite}
\usepackage{comment}
\usepackage{color}
\usepackage{amsmath}
\usepackage{amsthm}
\numberwithin{equation}{section}
\usepackage{mathtools}
\usepackage{stmaryrd}
\usepackage{cases}
\theoremstyle{definition}
\newtheorem*{thm*}{Theorem}
\newtheorem{thm}{Theorem}[section]
\newtheorem{prop}[thm]{Proposition}
\newtheorem{lem}[thm]{Lemma}
\newtheorem{cor}[thm]{Corollary}

\newtheorem{defn}[thm]{Definition}
\newtheorem*{pf1}{Proof of Theorem \ref{thm1}}
\newtheorem*{pf2}{Proof of Theorem \ref{thm2}}
\newtheorem{rmk}[thm]{Remark}

\newtheorem*{ack}{Acknowledgements}

\theoremstyle{plain}
\usepackage{amssymb}
\newcommand{\al}{\alpha}
\newcommand{\gam}{\gamma}
\newcommand{\del}{\delta}
\newcommand{\ep}{\varepsilon}
\newcommand{\kap}{\kappa}
\newcommand{\lam}{\lambda}
\newcommand{\sig}{\sigma}
\renewcommand{\phi}{\varphi}

\newcommand{\Del}{\Delta}

\newcommand{\Lam}{\Lambda}

\newcommand{\Om}{\Omega}

\newcommand{\N}{\mathbb{N}}

\newcommand{\R}{\mathbb{R}}

\newcommand{\Z}{\mathbb{Z}}

\newcommand{\cF}{\mathcal{F}}

\newcommand{\cK}{\mathcal{K}}

\newcommand{\cM}{\mathcal{M}}

\newcommand{\cS}{\mathcal{S}}

\newcommand{\fX}{\mathfrak{X}}

\renewcommand{\leq}{\leqslant}
\renewcommand{\geq}{\geqslant}
\newcommand{\pl}{\partial}
\newcommand{\norm}[1]{\left\|#1\right\|}
\newcommand{\cd}{\mathrel{}\middle|\mathrel{}}
\begin{document}
\title{Solvability of inhomogeneous fractional semilinear heat equations in Lorentz--Morrey spaces}
\author{Yusuke Oka}
\date{\today}

\maketitle

\begin{abstract}
  Of concern is the Cauchy problem for the fractional semilinear heat equation with inhomogeneous terms
  \begin{align*}
    \begin{cases}
      \pl_{t}u+(-\Del)^{\frac{\theta}{2}}u=|u|^{\gam -1}u+\mu, & \quad x\in\R^{N},\: t>0,\\
      u(x,0)\equiv0, & \quad x\in\R^{N},
    \end{cases}\tag{P}
  \end{align*}
  where $\theta>0$, $\gam>1,\: N\in\Z_{\geq 1}$ and $\mu$ is a tempered distribution on $\R^{N}$.
  By introducing the Lorentz--Morrey spaces,
  we overcome limitations of real interpolation in the classical local-Morrey spaces
  and obtain a sharp integral estimate for the nonlinear term.
  Moreover, in terms of Besov-type spaces,
  we give necessary conditions and sufficient conditions on $\mu$ for the local-in-time solvability of $\rm{(P)}$ in Lorentz--Morrey spaces.
\end{abstract}

\vspace{25pt}
\noindent Address:

\smallskip
\noindent
Graduate School of Mathematical Sciences, The University of Tokyo,\\
3-8-1 Komaba, Meguro-ku, Tokyo 153-8914, Japan.

\smallskip
\noindent
E-mail: {\tt oka@ms.u-tokyo.ac.jp}\\

\vspace{20pt}

\vspace{3pt}

\vspace{3pt}


\newpage
\section{Introduction and main results}\label{sec1}
We study the Cauchy problem for the fractional semilinear heat equation 
\begin{align}\label{eq}
  \begin{cases}
    \pl_{t}u+(-\Del)^{\frac{\theta}{2}}u=|u|^{\gam -1}u+\mu, & \quad x\in\R^{N},\: t>0,\\
    u(x,0)\equiv 0, & \quad x\in\R^{N},
  \end{cases}
  \tag{P}
\end{align}
where $\pl_{t}:=\pl/\pl t$, $\gam>1$, $\theta>0$, $N\in\Z_{\geq 1}$ and the inhomogeneous term $\mu=\mu(x)$ is a tempered distribution on $\R^{N}$.
The symbol $(-\Del)^{\frac{\theta}{2}}$ represents the fractional power of the Laplace operator $-\Del$ in $\R^{N}$, i.e.,
\begin{align*}
  (-\Del)^{\frac{\theta}{2}}f(x)
  =
  \cF^{-1}\left[
    |\xi|^{\theta}\cF f(\xi)
  \right](x),
  \quad x\in\R^{N}.
\end{align*}
Here $\cF$ and $\cF^{-1}$ denote the Fourier transform and its inverse on $\R^{N}$, respectively.

The aim of the present paper is twofold.
The first aim is to construct local-in-time solutions to problem \eqref{eq} for a wide class of inhomogeneous terms,
including distributions not previously treated,
such as derivatives of the Dirac delta and certain measurable functions that are not locally integrable.
Inhomogeneous Besov--Morrey spaces are well-suited to this end;
however, due to the insufficiency of real interpolation in the classical Morrey spaces,
one cannot expect sharp results.
We therefore introduce the Lorentz--Morrey spaces and verify their properties.
Second, by establishing a new estimate on Besov-type spaces,
we provide a necessary condition on the inhomogeneous term $\mu$ for the existence of local-in-time (possibly) sign-changing solutions belonging to Lorentz--Morrey spaces.

From the perspective of solvability, problem \eqref{eq} has been studied by many mathematicians (see e.g., \cite{BLZ2000, HIT2020, IKT2024, KT2017, KK2004, KK2005, Zhen2007, Zhang1998, Zhang1999} and references therein).
Among others, spatial singularities of inhomogeneous terms were considered in \cite{HIT2020, IKT2024},
and sufficient conditions and necessary conditions on the inhomogeneous terms for the local-in-time solvability of \eqref{eq} have been established in terms of local growth rates and function spaces.
Hereafter, we write
\[
  \gam_{\ast}:=\frac{N}{N-\theta}\,\,\text{ if }\,\,\theta\in\left]0,N\right[
  \quad\text{and}\quad
  \gam_{\ast}:=\infty\,\,\text{ if }\,\,\theta\geq N.
\]

We first review previous work on sufficient conditions for the existence of solutions to \eqref{eq}, and then state our objective.
Hisa, Ishige and Takahashi \cite{HIT2020} demonstrated some solvable criteria for $\mu$ in the case of $\theta\leq 2$.
Their results on sufficient conditions can be summarized as follows.
There exists a constant $A>0$ depending on $N$, $\theta$ and $\gam$ such that the following statements hold.
\begin{enumerate}[\rm(A)]
  \item\label{HIT-2} Let $\gam<\gam_{\ast}$.
  If a nonnegative Radon measure $\mu$ in $\R^{N}$ satisfies
  \begin{align*}
    \sup_{x\in\R^{N}}
    \mu(B(x,T^{1/\theta}))\leq
    AT^{\frac{N}{\theta}-\frac{\gam}{\gam-1}}
  \end{align*}
  for some $T>0$,
  then problem \eqref{eq} admits a solution in $\R^{N}\times\left[0,T\right[$.
  \item\label{HIT-1} Let $\gam\geq\gam_{\ast}$.
  If a nonnegative measurable function $\mu:\R^{N}\to\R$ satisfies
  \begin{align*}
    \mu(x)
    \leq
    \begin{cases}
      A|x|^{-\frac{\theta\gam}{\gam-1}}+C_{0}
      &\text{if }\gam>\gam_{\ast},
      \\
      A|x|^{-N}\left[\log\left(e+\frac{1}{|x|}\right)\right]^{-\frac{N}{\theta}}+C_{0}
      &\text{if }\gam=\gam_{\ast},
    \end{cases}
  \end{align*}
  for some $C_{0}>0$,
  then problem \eqref{eq} admits a local-in-time solution.
\end{enumerate}
Statement (\ref{HIT-2}) was shown by the contraction mapping theorem in uniformly local $L^{\gam}$ spaces,
while (\ref{HIT-1}) was obtained by constructing of supersolutions.

Ishige, Kawakami and Takada \cite{IKT2024} found sufficient conditions on $\mu$ for the case of $\gam\geq\gam_{\ast}$ and $\theta\leq 2$ by combining the real interpolation method with the uniformly local Zygmund type spaces,
which was recently introduced in \cite{IIK2025} and applied to Fujita critical Cauchy problem (\eqref{eq:initial} below with $\gam=1+\theta/N$).
In light of the statement (\ref{HIT-1}), their sufficient conditions are sharp with respect to the integrability index of function spaces.
However,
their function spaces do not take the effect of sign-changing of functions into account adequately.

In this paper, we adopt Besov--Lorentz--Morrey (BLM) spaces (see Definition \ref{def:iBL}) as the space of inhomogeneous terms.
It is well known that Morrey spaces do not fit well with the real interpolation,
but we can recover it by introducing the Lorentz spaces as a base of Morrey spaces (see Proposition \ref{cor-LRI}).
We note that this is established through a discussion of block spaces in \cite{Fe2016}.
Here, we give a direct proof,
which enables us to obtain more general versions of real interpolation for Morrey-type spaces,
including a local Lorentz--Morrey version.
As a consequence, we can obtain a sharp sufficient condition for a wide range of inhomogeneous terms (see Theorem \ref{thm1}).
In the case of $\gam>\gam_{\ast}$,
our result covers previous studies \cite{HIT2020, IKT2024}
and takes advantage of Morrey spaces, allowing us to handle a rich class of $\mu$
(see Remark \ref{rmk1}, \ref{rmk3} and \ref{rmk4}).
Our existence theorem also works well in the case of $\gam<\gam_{\ast}$;
derivatives of the Dirac delta are dealt under suitable conditions on $\theta$ and $\gam$,
as well as the sufficient condition in \cite{HIT2020} is covered (see Remark \ref{rmk2}).

\medskip

Secondly, we tackle necessary conditions for the existence of solutions.
Hisa, Ishige and Takahashi \cite{HIT2020} (building on \cite{HisaIshige2018, KK2004, KK2005}) found a necessary condition on Radon measure $\mu$ for the existence of nonnegative local-in-time solutions to \eqref{eq}.
If $u$ is a nonnegative solution to problem \eqref{eq} on $\left]0,T\right[$ for some $T>0$,
then the inhomogeneous term $\mu$ is a Radon measure and satisfies the following local growth conditions:
There exist constants $A_{0}$, $A_{1}>0$ such that the estimate
\begin{align*}
  \sup_{x\in\R^{N}}\mu(B(x,\sig))
  \leq A_{0}\sig^{N-\frac{\theta\gam}{\gam-1}}
  \quad\text{for }\sig\in\left]0,T^{1/\theta}\right]
\end{align*}
holds.
In addition, when $\gam=\gam_{\ast}$, it holds that
\begin{align*}
  \sup_{x\in\R^{N}}\mu(B(x,\sig))
  \leq A_{1}
  \left[
    \log\left(
      e+\frac{T^{1/\theta}}{\sig}
    \right)
  \right]^{1-\frac{N}{\theta}}
  \quad\text{for }\sig\in\left]0,T^{1/\theta}\right].
\end{align*}
These conditions are sharp, considering their sufficient conditions (Claims (\ref{HIT-2}) and (\ref{HIT-1})).
We remark that their argument strongly depends on the nonnegative assumption on $u$ and thus we cannot apply their theory to sign-changing solutions.

Recently, in the study of the Cauchy problem of the Navier--Stokes equation in $\R^{N}$,
Kozono, Okada and Shimizu \cite{KOS2020, KOS2021} proved that if there exists a Serrin class solution,
then the initial data necessarily belongs to a specific homogeneous/inhomogeneous Besov spaces $B^{-1+N/p,q}_{p}$
(moreover they proved the existence of Serrin class solutions for initial data which belong to $B^{-1+N/p,q}_{p}$).
The key tool of their study is the characterization of homogeneous/inhomogeneous Besov spaces via heat semigroup (see e.g., \cite{BCDbook, Bui1984, Tri1988}).
It is noteworthy that positivity of $u$ is not assumed in their results.
See also \cite{FGH2016}.

Inspired by their works, we obtain necessary conditions on inhomogeneous terms for the existence of local-in-time solutions to \eqref{eq} in Lorentz--Morrey spaces (see Theorem \ref{thm2}).
Although our approach requires solutions to belong to specific function spaces,
which seems to be a strong assumption compared to \cite{HIT2020},
the positivity assumption on $u$ is not needed.
This feature serves as a comparative advantage over \cite{HIT2020}.
Since the manner in which distributions $\mu$ appear in \eqref{eq} is substantially different from that in \eqref{eq:initial} (cf. \eqref{inteq} and \eqref{eq:thm-01}),
a new estimate for the Besov-type norm of $\mu$ needs to be established (see Proposition \ref{prop0}).

Throughout the following, the letter $C$ denotes a generic positive constant, which may change from line to line.
\subsection{Function spaces and definition of solutions}
We introduce some function spaces.
Henceforth, the symbol $\norm{f\cd X}$ denotes the quasi-norm of $f$ in any quasi-normed space $X$.

First we define the spaces for our solutions to problem \eqref{eq}.
For any mesurable function $f$, we define a function $\al_{f}:\R_{> 0}\to\left[0,\infty\right]$ as
\begin{align*}
  \al_{f}(\sig):=
  \left|
    \left\{
      x\in\R^{N}\mid |f(x)|>\sig
    \right\}
  \right|
\end{align*}
and a function $f^{\ast}:\R_{>0}\to\left[0,\infty\right]$ as
\begin{align*}
  f^{\ast}(\lam)
  &:=
  \sup
  \left\{
    \sig\in\R_{>0}
    \cd
    \al_{f}(\sig)
    >\lam
  \right\}.
\end{align*}
Here we adopt the convention $\sup\varnothing=0$,
and $|E|$ denotes the $N$-dimensional Lebesgue measure of a measurable set $E$.
\begin{defn}
  { (Lorentz spaces and Lorentz--Morrey spaces)}
  \begin{enumerate}
    \item
    Let $p\in\left]0,\infty\right[$ and $q\in\left]0,\infty\right]$.
    The Lorentz space $L^{p,q}(\R^{N})$ is defined as the set of measurable functions $f:\R^{N}\to\R$ such that $\norm{f\cd L^{p,q}}<\infty$,
    \begin{align*}
      \text{where }
      \norm{f\cd L^{p,q}}
      :=
      \begin{dcases}
        \left(
          \int_{0}^{\infty}
          \left(
            \lam^{\frac{1}{p}}
            f^{\ast}(\lam)
          \right)^{q}
          \,\frac{d\lam}{\lam}
        \right)^{1/q}
        &\text{for }q\in\left]0,\infty\right[,
        \\
        \sup_{\lam>0}\,
        \lam^{\frac{1}{p}}
        f^{\ast}(\lam)
        =\sup_{\lam>0}\lam\al_{f}(\lam)^{1/p}
        &\text{for }q=\infty.
      \end{dcases}
    \end{align*}
    \item
    Let $0<q\leq p<\infty$ and $\kap\in\left]0,\infty\right]$.
    The (local) Lorentz--Morrey space $M^{p}_{q,\kap}(\R^{N})$ is defined
    as the set of measurable functions $f:\R^{N}\to\R$ such that
    \begin{align*}
      \norm{f\cd M^{p}_{q,\kap}}
      &:=
      \sup_{0<R\leq 1}\sup_{z\in\R^{N}}
      R^{\frac{N}{p}-\frac{N}{q}}
      \norm{f\chi_{B(z,R)}\cd L^{q,\kap}(\R^{N})}
      <\infty,
    \end{align*}
    where $\chi_{E}$ is the characteristic function of $E(\subset\R^{N})$ and $B(x,r)$ is the open ball in $\R^{N}$ with center $x\in\R^{N}$ and radius $r>0$.
  \end{enumerate}
\end{defn}
\begin{rmk}
  \,
  \begin{enumerate}
    \item
    Lorentz spaces $L^{p,q}$ are just quasi-normed spaces, but for $p\in\left]1,\infty\right[$ and $q\in\left[1,\infty\right]$, they can be equipped with an equivalent norm (see Remark \ref{rmk2-1}).
    Thus, we regard the spaces $\left(L^{p,q},\norm{\cdot\cd L^{p,q}}\right)$ as Banach spaces by abuse of notation,
    in the case of $p>1$ and $q\geq 1$.
    \item
    Choosing $q$ and $r$ so that the Lorentz space $L^{q,r}$ is a Banach space,
    we can prove that the local Lorentz--Morrey space $M^{p}_{q,r}$ is also a Banach space.
    \item Fix $0<q\leq p<\infty$.
    Since it follows from the identity
    $\al_{f^{\ast}}=\al_{f}$ and the layer-cake representation
    that $L^{q,q}=L^{q}$, we define the local Morrey spaces by
    $M^{p}_{q}:=M^{p}_{q,q}$.
    It is a Banach space if $q\geq 1$.
  \end{enumerate}
\end{rmk}

We introduce the inhomogeneous Littlewood--Paley dyadic decomposition in order to define the inhomogeneous Besov-type spaces.
Let $\zeta(t)\in C_{c}^{\infty}\left(\left[0,\infty\right[;\R\right)\:$ satisfy $\:0\leq\zeta(t)\leq 1,\:\zeta(t)\equiv 1$ for $t\leq 3/2$ and $\text{supp}\:\zeta\subset\left[0,5/3\right[$.
We set
\begin{align*}
  \phi_{(0)}(\xi):=\zeta(|\xi|),\quad
  \phi_{0}(\xi):=\zeta(|\xi|)-\zeta(2|\xi|),\quad
  \phi_{j}(\xi):=\phi_{0}(2^{-j}\xi)\quad\text{for}\,\, j\in\Z.
\end{align*}
Then we have $\phi_{(0)}(\xi),\phi_{j}(\xi)\in C_{c}^{\infty}(\R^N)$ for all $j\in\Z$ and
\begin{align*}
  \phi_{(0)}(\xi)+\sum_{j=1}^{\infty}\phi_{j}(\xi)=\lim_{n\rightarrow\infty}\zeta(2^{-n}|\xi|)=1\quad\text{for all }\xi\in\R^{N}.
\end{align*}
For $f\in\cS'(\R^{N})$ and $j\in\N:=\left\{0,1,2,...\right\}$, we write
\begin{align*}
  \Del_{j}f
  :=
  \begin{cases}
    \cF^{-1}\left[\phi_{(0)}(\xi)\cF f\right]
    &\text{for }j=0,
    \\
    \cF^{-1}\left[\phi_{j}(\xi)\cF f\right]
    &\text{for }j\geq 1.
  \end{cases}
\end{align*}
Now we define the spaces of inhomogeneous terms.
\begin{defn}\label{def:iBL}
  {(inhomogeneous BLM spaces)}

  Let $1<q\leq p<\infty$, $r,\kap\in\left[1,\infty\right]$ and $s\in\R$.
  The inhomogeneous Besov--Lorentz--Morrey (BLM) space $B^{s,r}_{p,q,\kap}(\R^N)$ is defined as the set of tempered distributions $f\in\cS'(\R^N)$ such that
  \begin{align*}
    \Del_{j}f&\in M^{p}_{q,\kap}(\R^{N})
    \quad\text{for all }j\in\N\quad\text{and }
    \\
    \norm{f\mid B^{s,r}_{p,q,\kap}}&\coloneqq
    \norm{\left\{2^{sj}\norm{\Del_{j}f
    \mid M^{p}_{q,\kap}}\right\}_{j\in\N}\cd l^{r}(\N)}
    <\infty.
  \end{align*}
\end{defn}
We also let $B^{s,r}_{p}$ and $B^{s,r}_{p,q}$ denote the usual Besov spaces and Besov--Morrey spaces (defined in the same way as inhomogeneous BLM spaces by replacing $M^{p}_{q,\kap}$ with $L^p$ and $M^{p}_{q}$), respectively.
This kind of generalization of Besov spaces was introduced by Kozono and Yamazaki \cite{KozonoYamazaki} and has been effectively used in the study of semilinear parabolic equations (see e.g., \cite{OZ2024, Zhanpeisov23}).
One can consult \cite{LRbook, Tri1} for general references.
We mainly use spaces $B^{s,r}_{p,q,\kap}$ for the case of $r\in\left\{1,\infty\right\}$ and $\kap=\infty$.
Next we put $S(t)f:=\cF^{-1}\left[\exp\left(-t|\xi|^{\theta}\right)\cF f\right]$
for every $f\in\cS'(\R^{N})$ and $t>0$ to formulate our definition of solutions.
\begin{defn}\label{def:1-2}
  Let $1<b\leq a<\infty$, $s\in\R$, $q,r,c\in\left[1,\infty\right]$ and $T>0$.
  For $\mu\in B^{s,r}_{a,b,c}(\R^{N})$,
  we say that a measurable function $u:\R^{N}\times\left]0,T\right[\to\R$ is a {\it solution} to \eqref{eq} on $\left]0,T\right[$
  if $u$ satisfies
  \begin{align*}
    u\in L^{\infty}(0,T;M^{p}_{q,\infty}(\R^{N}))
  \end{align*}
  for some $1<q\leq p<\infty$
  and the following integral equation
  \begin{align}\label{inteq}
    u(x,t)
    =
    \int_{0}^{t}
    S\left(t-\tau\right)
    \left[
      \left|u\right|^{\gam-1}u(\cdot,\tau)
    \right](x)
    \,d\tau
    +
    \int_{0}^{t}
    S\left(t-\tau\right)
    \mu(x)
    \,d\tau
  \end{align}
  for a.e.-$(x,t)\in\R^{N}\times\left]0,T\right[$.
\end{defn}
There is a possibility that the second term of the right hand side of \eqref{inteq} is not a function if $\mu$ is a general tempered distribution.
In our existence theorem (Theorem \ref{thm1}),
such cases are eliminated and it certainly exists as an element of $L^{\infty}\left(0,T; M^{p}_{q,\infty}(\R^{N})\right)$.
We do not consider about initial condition in Definition \ref{def:1-2} but we make a comment on it (see Remark \ref{rmk3-5}).
\subsection{Main results}

First we state sufficient conditions for local-in-time existence.
\begin{thm}\label{thm1}
  Let $\gam>1$, $\gam<q\leq p$ and $p\geq{N(\gam-1)}/{\theta}$.
  \vspace{-5pt}
  \begin{enumerate}[{\rm (1)}]
    \item\label{thm1-1} 
    There exist positive constants $\del$, $M$ and $T$ such that for every $\mu\in\cS'(\R^{N})$ satisfying
    $\norm{\mu\cd B^{-\theta,1}_{p,q,\infty}}<\del$,
    problem \eqref{eq} possesses a solution $u$ on $\left]0,T\right[$ with a bound
    \begin{align}\label{thm-bd}
      \sup_{t\in\left]0,T\right[}
      \norm{u(\cdot,t)
      \cd M^{p}_{q,\infty}(\R^{N})}\leq M.
    \end{align}
    \item\label{thm1-2} 
    There exist $\del>0$ and $M>0$ such that for every
    $\mu\in\cS'(\R^{N})$ satisfying
    \begin{align*}
      \mu\in B^{\ep-\theta,\infty}_{\frac{Np}{N+p\ep},\frac{Nq}{N+p\ep},\infty}(\R^{N})
      \quad\text{and}\quad
      \limsup_{j\to\infty}
      2^{(\ep-\theta)j}\norm{\Del_{j}\mu\cd M^{\frac{Np}{N+p\ep}}_{\frac{Nq}{N+p\ep},\infty}(\R^{N})}<\del
    \end{align*}
    for some $\ep\in\left]0,\min\left\{\theta,N(q-1)/p\right\}\right[$,
    problem \eqref{eq} possesses a solution $u$ on $\left]0,T\right[$ for some $T>0$ with a bound
    \eqref{thm-bd}.
    \item\label{thm1-3} Assume
    $-\theta<s<0$.
    There exists $M>0$ such that for every
    $\mu\in B^{s,\infty}_{p,q,\infty}(\R^{N})$,
    problem \eqref{eq} possesses a solution $u$ on $\left]0,T\right[$ for some $T>0$ with a bound
    \eqref{thm-bd}.
    \qed
  \end{enumerate}
\end{thm}
Each of these three sufficient conditions on $\mu$ has its own advantages.
Note that $p={N(\gam-1)}/{\theta}$ is feasible only when the Serrin supercritical condition $\gam>\gam_{\ast}$ holds,
because we imposed $p\geq q>\gam$.
\begin{rmk}\label{rmk1}
  Let us consider the Serrin supercritical case $\gam>\gam_{\ast}={N}/{(N-\theta)}$ with $\theta<N$.
  From Proposition \ref{prop:2-2} and Proposition \ref{prop:SE}, we have
  \begin{align*}
    M^{\frac{N(\gam-1)}{\gam\theta}}_{{q}/{\gam}}
    \subset
    M^{\frac{N(\gam-1)}{\gam\theta}}_{{q}/{\gam},\infty}
    \subset
    B^{0,\infty}_{\frac{N(\gam-1)}{\gam\theta},\frac{q}{\gam},\infty}
    \subset
    B^{\ep-\theta,\infty}_{\frac{N(\gam-1)}{\theta+(\gam-1)\ep},\frac{Nq}{N+p\ep},\infty}
  \end{align*}
  for some $\ep>0$.
  Thus,
  by applying Theorem \ref{thm1} (\ref{thm1-2}) with $p={N(\gam-1)}/{\theta}$ and any $q\in\left]\gam,p\right[$,
  we find out that
  problem \eqref{eq} possesses a solution for $\mu\in M^{{N(\gam-1)}/{\gam\theta},\infty}_{q/\gam}$ satisfying $\norm{\mu\cd M^{{N(\gam-1)}/{\gam\theta},\infty}_{q/\gam}}<\del$.
  For example, we can take homogeneous functions $c|x|^{-\frac{\gam\theta}{\gam-1}}$ or their infinite sum
  \begin{align*}
    c\sum_{j\in\N}|x-x_{j}|^{-\frac{\gam\theta}{\gam-1}}\chi_{B(x_{j},1)}(x),\quad |x_{j}-x_{k}|>1\text{ for }j\neq k
  \end{align*}
  as inhomogeneous terms for sufficiently small $c>0$.
  This is compatible with the results of \cite{HIT2020, IKT2024}.
  See Remark \ref{rmk4} and Proposition \ref{prop:B-1} for further examples;
  therein lies the advantage of employing Morrey-type spaces.
\end{rmk}
\begin{rmk}\label{rmk3}
  Assume $\gam>\gam_{\ast}$.
  We know that all $\mu\in L^{{N(\gam-1)}/{\gam\theta}}(\R^N)$ are acceptable as inhomogeneous terms.
  Recall that problem \eqref{eq} has the following structure:
  If a function $u(x,t)$ is a solution to problem \eqref{eq} with a inhomogeneous term $\mu(x)$,
  then a function $u_{\lam}(x,t):=\lam^{\frac{\theta}{\gam-1}}u(\lam x,\lam^{\theta}t)$ is a solution to problem \eqref{eq} with $\mu_{\lam}(x):=\lam^{\frac{\gam\theta}{\gam-1}}\mu(\lam x)$.
  Let $p=N(\gam-1)/\theta$.
  For any $\mu\in L^{{N(\gam-1)}/{\gam\theta}}(\R^N)$,
  we see that
  \begin{align*}
    \limsup_{j\to\infty}2^{(\ep-\theta)j}
    \norm{\Del_{j}\mu_{\lam}\cd M^{\frac{Np}{N+p\ep}}_{\frac{Nq}{N+p\ep},\infty}}
    &\leq
    \norm{\mu_{\lam}\cd B^{\ep-\theta,\infty}_{\frac{Np}{N+p\ep},\frac{Nq}{N+p\ep},\infty}}
    \leq
    C\norm{\mu_{\lam}\cd B^{0,\infty}_{\frac{N(\gam-1)}{\gam\theta},\frac{q}{\gam},\infty}}
    \\
    \leq
    C\norm{\mu_{\lam}\cd M^{\frac{N(\gam-1)}{\gam\theta}}_{\frac{N(\gam-1)}{\gam\theta},\infty}}
    &\leq
    C\sup_{z\in\R^{N}}
    \left(\int_{|x-z|\leq 1}
    \lam^{N}|\mu(\lam x)|^{\frac{N(\gam-1)}{\gam\theta}}\,dx
    \right)^{\frac{\gam\theta}{N(\gam-1)}}
    \\
    &=C
    \sup_{z\in\R^{N}}
    \left(
      \int_{|x-z|\leq\lam}
      |\mu(x)|^{\frac{N(\gam-1)}{\gam\theta}}
      \,dx
    \right)^{\frac{\gam\theta}{N(\gam-1)}}
  \end{align*}
  holds for all $\lam>0$.
  Then,
  by absolute continuity of $\left|\mu\right|^{\frac{N(\gam-1)}{\gam\theta}}$,
  the last term tends to zero if we send $\lam\to 0$.
  Thus we can apply Remark \ref{rmk1} to $\mu_{\lam}$, $0<\lam\ll 1$.
\end{rmk}
\begin{rmk}\label{rmk4}
  Assume $\theta>1$ and $\gam>\gam_{\ast}$.
  Set $\del:=\frac{1+(\theta-1)\gam}{N(\gam-1)}$
  and take any $q\in\left]\gam,{N(\gam-1)}/{\theta}\right[$.
  These conditions give $1<\frac{\theta q}{1+(\theta-1)\gam}<1/\del$.
  Then \cite[Proposition 1.3]{KozonoYamazaki}, \cite[Lemma 1.4]{KozonoYamazaki} and Proposition \ref{prop:2-2} imply
  $$\mu(x):=\prod_{j=1}^{N}\left(x_{j}\right)_{+}^{-\del}\in M^{1/\del}_{\theta q/(1+(\theta-1)\gam)}\subset B^{0,\infty}_{\frac{1}{\del},\frac{\theta q}{1+(\theta-1)\gam},\infty}$$
  and thus we see that
  $$\frac{1}{-\del}\frac{\pl\mu}{\pl x_{1}}=
  \mathrm{f.p.}\,(x_{1})_{+}^{-\del-1}\prod_{j=2}^{N}(x_{j})_{+}^{-\del}\in
  B^{-1,\infty}_{\frac{1}{\del},\frac{\theta q}{1+(\theta-1)\gam},\infty}$$
  holds.
  Since the space $B^{-1,\infty}_{{1}/{\del},{\theta q}/{\left(1+(\theta-1)\gam\right)},\infty}$ satisfies the assumptions of Theorem \ref{thm1} (\ref{thm1-2})
  with $\ep=\theta-1$ and $p=N(\gam-1)/\theta$,
  it turns out that the distribution $c(\pl\mu/\pl x_{1})$
  can be taken as the inhomogeneous term if $0<c\ll 1$.
  We also note that $\mu$ does not belong to Lorentz spaces if $N\geq 2$.
\end{rmk}
\begin{rmk}\label{rmk2}
  In the Serrin subcritical case $\gam<\gam_{\ast}$, we can pick $p_{0}$ as $\gam<p_{0}<\gam_{\ast}$.
  Then by Proposition \ref{prop:SE}, we have
  \begin{align*}
    L^{1}\subset B^{0,\infty}_{1}
    \subset B^{N/p_{0}-N,\infty}_{p_{0}}
    \subset B^{N/p_{0}-N,\infty}_{p_{0},p_{0},\infty}
  \end{align*}
  and the space appearing in the last satisfies the assumptions of Theorem \ref{thm1} (\ref{thm1-3}).
  For instance, we know that the Dirac delta measure $\del(x)$ belongs to $B^{0,\infty}_{1}\subset B^{N/p_{0}-N,\infty}_{p_{0}}
  \subset B^{N/p_{0}-N,\infty}_{p_{0},p_{0},\infty}$.
  Furthermore, if $\theta>m$ with $m\in\Z_{\geq 1}$ and if $\gam>1$ satisfies
  $\gam<{N}/{(N+m-\theta)}$,
  we can take $\pl^{\al}\del(x)$ as an inhomogeneous term of problem \eqref{eq} since we can take $p_{1}\in\left]\gam,{N}/{(N+m-\theta)}\right[$ and we have
  \begin{align*}
    \pl^{\al}\del(x)\in
    B^{N/p_{1}-N-m,\infty}_{p_{1}}
    \subset
    B^{N/p_{1}-N-m,\infty}_{p_{1},p_{1},\infty}
    \quad\text{for }
    i\in\left\{1,...,N\right\}
    \,\,\text{and for }|\al|\leq m.
  \end{align*}
  This example reflects a benefit of employing Besov-type spaces.
\end{rmk}
In this paper, we do not treat in detail the Serrin critical case $\gam=\gam_{\ast}$. This case requires some special function spaces (see Theorem 1.3 in \cite{IKT2024} for details).

The next theorem states a necessary condition on inhomogeneous terms $\mu$
for the existence of a solution $u=u(x,t)$ belonging to $L^{\infty}(0,T;M^{p}_{q,\infty}(\R^{N}))$.
\begin{thm}\label{thm2}
  Let $T>0$, $p\geq q>\gam$ and $p\geq{N(\gam-1)}/{\theta}$.
  Assume that a function $u=u(x,t)\in L^{\infty}\left(0,T;M^{p}_{q,\infty}(\R^{N})\right)$ satisfies the integral equation \eqref{inteq} for a.e.-$(x,t)\in\R^{N}\times\left]0,T\right[$.
  Then the tempered distribution $\mu$ appearing in \eqref{inteq}
  must be an element of $B^{-\theta,\infty}_{p,q,\infty}(\R^{N})$.
  \qed
\end{thm}
\begin{rmk}
We claim nothing about existence of solutions to problem \eqref{eq} for $\mu\in\cS'(\R^{N})$ satisfying
\begin{align*}
  \mu\in B^{-\theta,\infty}_{p,q,\infty}\setminus
  \left\{
  B^{-\theta,1}_{p,q,\infty}
  \cup
  \left(
    \bigcup_{0<\ep<\theta}
    B^{\ep-\theta,\infty}_{\frac{Np}{N+p\ep},\frac{Nq}{N+p\ep},\infty}
  \right)
  \right\}.
\end{align*}
\end{rmk}
By similarly carrying the arguments of the proof of Theorem \ref{thm2} with the aid of Theorem 2 in \cite{BST} and of the Hardy--Littlewood--Sobolev inequality (see also \cite{KOS2021}), we can obtain a necessary condition on the initial data for the local-in-time solvability of the following Cauchy problem
\begin{align}\label{eq:initial}
  \begin{cases}
    \pl_{t}u+(-\Del)^{\theta/2}u=|u|^{\gam -1}u, & \quad x\in\R^{N},\: t>0,\\
    u(x,0)=\mu(x)\in\cS'(\R^{N}), & \quad x\in\R^{N},
  \end{cases}
\end{align}
where $\gam>1$, $\theta>0$.

\begin{cor}
  Let $p$ and $q$ satisfy the conditions
  \begin{align*}
    \frac{N(\gam-1)}{\theta}<p<\infty,\quad
    \gam\leq p,\quad
    \gam<q<\infty
    \quad\text{and}\quad
    \frac{\theta}{\gam-1}-\frac{\theta}{q}-\frac{N}{p}=0.
  \end{align*}
  Assume that a function $u(x,t)\in L^{q}(0,T;L^{p}(\R^{N}))$ satisfies the integral equation
  \begin{align}\label{eq:thm-01}
    u(x,t)
    =
    S({t})\mu(x)
    +
    \int_{0}^{t}S(t-\tau)
    \left[
      |u|^{\gam-1}u(\cdot,\tau)
    \right](x)\,d\tau
  \end{align}
  for a.e.-$(x,t)\in\R^{N}\times\left]0,T\right[$ for some $T>0$.
  Then the tempered distribution $\mu(x)$ appearing in \eqref{eq:thm-01} must be an element of inhomogeneous Besov space $B^{-\theta/q,q}_{p}(\R^{N})=B^{N/p-\theta/(\gam-1),q}_{Nq(\gam-1)/\theta(q-\gam+1)}(\R^{N})$.
  \qed
\end{cor}

The rest of this paper proceeds as follows.
We divide Section \ref{sec2} into five subsections.
The first three are devoted to establishing
basic properties, including real interpolation for Lorentz--Morrey spaces and Besov spaces.
We confirm decay estimates of operators $S(t)$ in the next subsection.
At the end of Section \ref{sec2},
we make an estimate of the BLM norm of the inhomogeneous term.
When $\theta\neq2$, it needs additional care and we postpone this task to Appendix \ref{sec-A}.
In Section \ref{sec3}, we prove Theorem \ref{thm1} by applying the contraction mapping theorem and Theorem \ref{thm2} by applying Proposition \ref{prop0}.
We make some comments on Morrey spaces in Appendix B.
\section{Preliminaries}\label{sec2}
In this section, we gather some preliminary properties of Lorentz spaces $L^{p,q}$, of Lorentz--Morrey spaces $M^{p}_{q,r}$ and of BLM spaces $B^{s,r}_{p,q,\kap}$.
We also confirm decay estimates of operators $S(t)$ on these spaces.
In the last subsection, we show that BLM norms $\norm{f\cd B^{-\theta,\infty}_{p,q,\kap}}$ can be estimated by a space-time norm of the Duhamel integral of $f$.
\subsection{Basic properties of Lorentz--Morrey spaces}
We begin with recalling basic properties of Lorentz spaces.
\begin{lem}\label{lem:Lorentz}
  Fix $p\in\left]0,\infty\right[$.
  The following claims hold.
  \begin{enumerate}[{\rm (1)}]
    \item\label{lem:Lorentz-1}
    Assume that $p_{0}$ and $p_{1}\in\left]0,\infty\right[$ satisfy
    $p^{-1}=p_{0}^{-1}+p_{1}^{-1}$.
    Then there exists $C>0$ depending on $p_{0}$, $p_{1}$ such that the inequality
    \begin{align*}
      \norm{fg\cd L^{p,\infty}}
      \leq C\norm{f\cd L^{p_{0},\infty}}
      \norm{g\cd L^{p_{1},\infty}}
    \end{align*}
    holds for all $f\in L^{p_{0},\infty}$ and $g\in L^{p_{1},\infty}$.
    \item\label{lem:Lorentz-2}
    Fix $r\in\left]0,\infty\right[$.
    Then the identity
    \begin{align*}
      \norm{|f|^{r}\cd L^{p,\infty}}
      =\norm{f\cd L^{pr,\infty}}^{r}
    \end{align*}
    holds for all $f\in L^{pr,\infty}$.
    \item\label{lem:Lorentz-3} Assume $|g(x)|\leq |f(x)|$ for a.e.-$x$.
    Then we have
    \begin{align*}
      \norm{g\cd L^{p,\infty}}\leq\norm{f\cd L^{p,\infty}}.
    \end{align*}
    \item\label{lem:Lorentz-4} We have the following scaling property.
    For all $\lam>0$ and $f\in L^{p,q}$, we have
    \begin{align*}
      \norm{f_{\lam}\cd L^{p,q}}
      =
      \lam^{\frac{-N}{p}}\norm{f\cd L^{p,q}},
    \end{align*}
    where $f_{\lam}(x):=f(\lam x)$ and $q\in\left]0,\infty\right]$.
  \end{enumerate}
\end{lem}
\begin{proof}
  (\ref{lem:Lorentz-1}) See Exercise 1.1.15 in \cite{GrafakosCBook}. Others are easily confirmed.
\end{proof}
Lemma \ref{lem:Lorentz} leads the corresponding properties of Lorentz--Morrey spaces.
\begin{lem}\label{lem:LM}
  Fix $0<q\leq p<\infty$.
  The following claims hold.
  \begin{enumerate}[{\rm (1)}]
    \item\label{lem:LM-1}
    Assume that $p_{0}$, $p_{1}$, $q_{0}$ and $q_{1}\in\left]0,\infty\right[$ satisfy
    $p_{0}>q_{0}$, $p_{1}>q_{1}$ and $\left(p^{-1},q^{-1}\right)=\left(p_{0}^{-1}+p_{1}^{-1},q_{0}^{-1}+q_{1}^{-1}\right)$.
    Then there exists $C>0$ depending on $p_{0}$, $p_{1}$ $q_{0}$ and $q_{1}$ such that the inequality
    \begin{align*}
      \norm{fg\cd M^{p}_{q,\infty}}
      \leq C\norm{f\cd M^{p_{0}}_{q_{0},\infty}}
      \norm{g\cd M^{p_{1}}_{q_{1},\infty}}
    \end{align*}
    holds for all $f\in M^{p_{0}}_{q_{0},\infty}$ and $g\in M^{p_{1}}_{q_{1},\infty}$.
    \item\label{lem:LM-2}
    Fix $r\in\left]0,\infty\right[$.
    Then the identity
    \begin{align*}
      \norm{|f|^{r}\cd M^{p}_{q,\infty}}
      =\norm{f\cd M^{pr}_{qr,\infty}}^{r}
    \end{align*}
    holds for all $f\in M^{pr}_{qr,\infty}$.
    \item\label{lem:LM-3} Assume $|g(x)|\leq |f(x)|$ for a.e.-$x$. Then we have
    \begin{align*}
      \norm{g\cd M^{p}_{q,\infty}}\leq\norm{f\cd M^{p}_{q,\infty}}.
    \end{align*}
  \end{enumerate}
\end{lem}
We prepare the Young inequality for Lorentz--Morrey spaces.
\begin{prop}\label{prop:YO-LM}
  Let $1<q\leq p<\infty.$
  There exists a positive constant $C$ depending on $q$ such that the inequality
  \begin{align*}
    \norm{g\ast f\cd M^{p}_{q,\infty}(\R^{N})}
    \leq C
    \norm{g\cd L^{1}(\R^{N})}
    \norm{f\cd M^{p}_{q,\infty}(\R^{N})}
  \end{align*}
  holds for all $g\in L^{1}(\R^{N})$ and $f\in M^{p}_{q,\infty}(\R^{N})$.
\end{prop}

Proposition \ref{prop:YO-LM} does not seem to directly follow from its counterpart of Lorentz spaces.
We provide a proof here for completeness;
the following argument is a modification of that on pp.~24--25 in \cite{GrafakosCBook}.

\begin{proof}[Proof of Proposition \ref{prop:YO-LM}]
  Fix any $R\in\left]0,1\right]$, $z\in\R^{N}$ and $\lam>0$.
  Take arbitrary $f\in M^{p}_{q,\infty}(\R^{N})$ and $g\in L^{1}(\R^{N})$.
  Set $M:=\lam/2\norm{g\cd L^{1}}$.
  We decompose $f$ as
  \begin{align*}
    f(x)=
    f(x)\chi_{\left\{|f(x)|\leq M\right\}}(x)
    +
    f(x)\chi_{\left\{|f(x)|>M\right\}}(x)
    =:
    f_{0}(x)+f_{1}(x).
  \end{align*}
  Because
  $\left|f_{0}\ast g(x)\right|\leq M\norm{g\cd L^{1}}=\lam/2$ holds for all $x\in\R^{N}$,
  we have
  \begin{align*}
    \al_{(f\ast g)\chi_{B(z,R)}}(\lam)
    \leq
    \al_{(f_{0}\ast g)\chi_{B(z,R)}}\left(\frac{\lam}{2}\right)
    +
    \al_{(f_{1}\ast g)\chi_{B(z,R)}}\left(\frac{\lam}{2}\right)
    =
    \al_{(f_{1}\ast g)\chi_{B(z,R)}}\left(\frac{\lam}{2}\right).
  \end{align*}
  By Chebyshev's inequality, we get
  \begin{align}\label{eq:yoloc1}
    \al_{(f\ast g)\chi_{B(z,R)}}(\lam)
    &\leq
    \al_{(f_{1}\ast g)\chi_{B(z,R)}}\left(\frac{\lam}{2}\right)
    \leq
    \frac{2}{\lam}
    \norm{f_{1}\ast g\cd L^{1}(B(z,R))}
    \notag
    \\
    &\leq
    \frac{2}{\lam}
    \int_{B(z,R)}\int_{\R^N}
    \left|f_{1}(x-y)\right||g(y)|
    \,dy\,dx
    \notag
    \\
    &=
    \frac{2}{\lam}
    \int_{\R^N}|g(y)|
    \int_{B(z-y,R)}
    \left|f_{1}(x)\right|
    \,dx\,dy.
  \end{align}
  Here, for fixed $y\in\R^{N}$, we carry out the following calculation.
  \begin{align}\label{eq:yoloc2}
    \int_{B(z-y,R)}
    \left|f_{1}(x)\right|
    \,dx
    &=
    \int_{B(z-y,R)}
    \int_{0}^{\infty}
    \chi_{\left\{0<k<|f_{1}(x)|\right\}}(k)
    \,dk\,dx
    =
    \int_{0}^{\infty}
    \al_{f_{1}\chi_{B(z-y,R)}}(k)
    \,dk
    \notag
    \\
    &=
    \int_{0}^{M}
    \al_{f\chi_{B(z-y,R)}}(M)
    \,dk
    +
    \int_{M}^{\infty}
    \al_{f\chi_{B(z-y,R)}}(k)
    k^{q}k^{-q}
    \,dk
    \notag
    \\
    &\leq
    M^{1-q}
    \norm{f\chi_{B(z-y,R)}\cd L^{q,\infty}}^{q}
    +
    \int_{M}^{\infty}
    \frac{dk}{k^{q}}
    \norm{f\chi_{B(z-y,R)}\cd L^{q,\infty}}^{q}
    \notag
    \\
    &\leq
    \frac{q}{q-1}M^{1-q}R^{N-\frac{Nq}{p}}
    \norm{f\cd M^{p}_{q,\infty}}^{q}.
  \end{align}
  We plug \eqref{eq:yoloc2} into \eqref{eq:yoloc1} and recall $M=\lam/2\norm{g\cd L^{1}}$ to see that
  \begin{align*}
    R^{\frac{N}{p}-\frac{N}{q}}\lam
    \al_{(f\ast g)\chi_{B(z,R)}}(\lam)^{1/q}
    &\leq
    2\left(
      \frac{q}{q-1}
    \right)^{1/q}
    \norm{g\cd L^{1}}
    \norm{f\cd M^{p}_{q,\infty}}
  \end{align*}
  holds for all $R\in\left]0,1\right]$, $z\in\R^{N}$ and $\lam>0$.
  Taking the sup in $\lam>0$ yields
  \begin{align*}
    R^{\frac{N}{p}-\frac{N}{q}}
    \norm{(f\ast g)\chi_{B(z,R)}\cd L^{q,\infty}}
    &\leq
    2\left(
      \frac{q}{q-1}
    \right)^{1/q}
    \norm{g\cd L^{1}}
    \norm{f\cd M^{p}_{q,\infty}}
  \end{align*}
  for all $(z,R)\in\R^{N}\times\left]0,1\right]$, and thus we arrive at the claim.
\end{proof}
\subsection{Real interpolation of Lorentz--Morrey spaces}
Next we introduce some notations for real interpolation.
We can pick the following from \cite{BLbook}.
\begin{defn}\label{def2-1}
  Let $X_{0}$ and $X_{1}$ be quasi-normed spaces and suppose that there exists a Hausdorff topological vector space $\fX$ such that $X_{0}$ and $X_{1}$ are subspaces of $\fX$.
  For $f\in X_{0}+X_{1}$ and $\lam>0$, we put
  \begin{align*}
    K(f, \lam; X_{0}, X_{1}):=
    \inf_{\substack{f=f_{0}+f_{1}\\f_{0}\in X_{0},\,f_{1}\in X_{1}}}
    \big\{
    \norm{f_{0}\cd X_{0}}
    +
    \lam\norm{f_{1}\cd X_{1}}
    \big\}.
  \end{align*}
  Then for $\kap\in\left]0,1\right[$ and $q\in\left]0,\infty\right]$, we define the space $\left(X_{0},X_{1}\right)_{\kap,q}$ as
  \begin{align*}
    \left(X_{0},X_{1}\right)_{\kap,q}
    :=
    \left\{
      f\in X_{0}+X_{1}\cd
      \norm{f\cd\left(X_{0},X_{1}\right)_{\kap,q}}
      <\infty
    \right\},
  \end{align*}
  where
  \begin{align*}
    \norm{f\cd\left(X_{0},X_{1}\right)_{\kap,q}}
    :=
    \begin{dcases}
      \left(
        \int_{0}^{\infty}
        \left(\lam^{-\kap}K(f,\lam;X_{0},X_{1})\right)^{q}
        \,\frac{d\lam}{\lam}
      \right)^{1/q}
      &\text{for }q\in\left]0,\infty\right[,
      \\
      \sup_{\lam>0}\,
      \lam^{-\kap}K(f,\lam;X_{0},X_{1})
      &\text{for }q=\infty.
    \end{dcases}
  \end{align*}
\end{defn}
\begin{prop}\label{prop2-1}
  Take $X_{0}$ and $X_{1}$ as in Definition \ref{def2-1}.
  Let $\kap\in\left]0,1\right[$, $q\in\left]0,\infty\right]$.
  \begin{enumerate}[{\rm (1)}]
    \item If $X_{0}$ and $X_{1}$ are normed spaces and $q\geq 1$, then $(X_{0},X_{1})_{\kap,q}$ is a normed space.
    \item If $X_{0}$ and $X_{1}$ are complete, then $(X_{0},X_{1})_{\kap,q}$ is complete.
  \end{enumerate}
\end{prop}
\begin{proof}
    (1) \, It can be easily checked.
    \quad (2) \, See Theorem 3.4.2 in \cite{BLbook}.
\end{proof}
\begin{thm}\label{LorentzRI}
  Let $p_{0}$ and $p_{1}\in\left]0,\infty\right[$ satisfy $p_{0}\neq p_{1}$ and take $q\in\left]0,\infty\right]$ and $\kap\in\left]0,1\right[$.
  We set $p\in\left]0,\infty\right[$ by the formula $p^{-1}=(1-\kap)p_{0}^{-1}+\kap p_{1}^{-1}$.
  Then we have $\left(L^{p_{0},q_{0}},L^{p_{1},q_{1}}\right)_{\kap,q}=L^{p,q}$ for all $q_{0}, q_{1}\in\left]0,\infty\right]$.
\end{thm}
\begin{proof}
  See Theorem 5.3.1 in \cite{BLbook}.
\end{proof}
\begin{rmk}\label{rmk2-1}
  If $1<p<\infty$ and $1\leq q\leq\infty$, we may write the space $L^{p,q}$ as
  \begin{align*}
    L^{p,q}
    =\left(L^{p_{0},p_{0}},L^{p_{1},p_{1}}\right)_{\kap,q}
    =\left(L^{p_{0}},L^{p_{1}}\right)_{\kap,q}
  \end{align*}
  with $p_{0},p_{1}\in\left]1,\infty\right[$, by means of Theorem \ref{LorentzRI}.
  Thus, Proposition \ref{prop2-1} (1) and (2) identify the space $L^{p,q}$ as a Banach space.
\end{rmk}
We prepare a real interpolation result for Lorentz--Morrey spaces.
\begin{prop}\label{cor-LRI}
  Let $\left(p_{0},p_{1},q_{0},q_{1}\right)\in\left]0,\infty\right[^{4}$ satisfy $p_{0}\neq p_{1}$, $q_{0}\neq q_{1}$,
  $q_{0}\leq p_{0}$ and $q_{1}\leq p_{1}$.
  We take $r\in\left]0,\infty\right]$ and $\kap\in\left]0,1\right[$,
  and set $p$, $q$ by the formula
  \begin{align*}
    \left(p^{-1},q^{-1}\right)=(1-\kap)\left(p_{0}^{-1},q_{0}^{-1}\right)+\kap\left(p_{1}^{-1},q_{1}^{-1}\right).
  \end{align*}
  Then $q\leq p$ holds and the embedding
  \begin{align*}
    \left(
        M^{p_{0}}_{q_{0},r_{0}}, M^{p_{1}}_{q_{1},r_{1}}
    \right)_{\kap,r}
    \subset
    M^{p}_{q,r}
  \end{align*}
  is valid
  for all $r_{0}$, $r_{1}\in\left]0,\infty\right]$.
\end{prop}
This property has a key role in section \ref{sec3}.
It is noteworthy that this kind of embedding fails in classical Morrey spaces
(see Theorem 3 (v) in \cite{LR2013}).
We prove this property in a more general form,
for which we introduce the following definition.
\begin{defn}
{($X$-Morrey spaces)}
Let $X$ be a quasi-normed function space and $p\in\left]0,\infty\right[$.
The local $X$-Morrey space $M^{p}_{X}$ is defined as the set of functions $f\in X$ such that $\norm{f\cd M^{p}_{X}}<\infty$,
where
\begin{align*}
  \norm{f\cd M^{p}_{X}}
  :=
  \sup_{0<R\leq 1}\sup_{z\in\R^{N}}
  R^{\frac{N}{p}}\norm{\left(f_{R}\right)\chi_{B(z,1)}\cd X}.
\end{align*}
Here, we set $f_{\lam}(x):=f(\lam x)$ for all $\lam>0$.
\end{defn}
Lemma \ref{lem:Lorentz}.(\ref{lem:Lorentz-4}) suggests that this definition is consistent with our Lorentz--Morrey spaces.
Note that $\left(f_{R}\right)\chi_{B(z,1)}=\left(f\chi_{B(Rz,R)}\right)_{R}$ holds for all $R>0$ and $z\in\R^{N}$.
\begin{prop}\label{XMoRI}
  Let $(p_{0},p_{1})\in\left]0,\infty\right[^{2}$
  satisfy $p_{0}\neq p_{1}$.
  We take $r\in\left]0,\infty\right]$ and $\kap\in\left]0,1\right[$,
  and set $p$ by the formula
  $p^{-1}=(1-\kap)p_{0}^{-1}+\kap p_{1}^{-1}$.
  Assume that quasi-normed function spaces $X$, $X_{0}$, $X_{1}$ have the property
  $
    \left(X_{0},X_{1}\right)_{\kap,r}
    \subset X.
  $
  Then the embedding
  $
    \left(M^{p_{0}}_{X_{0}},M^{p_{1}}_{X_{1}}\right)_{\kap,r}
    \subset M^{p}_{X}
  $
  is valid.
\end{prop}
\begin{proof}[Proof of Proposition \ref{XMoRI}]
  We restrict ourselves to the case $r=\infty$ (the remaining cases are analogous).
  Fix $z\in\R^{N}$ and $R\in\left]0,1\right]$.
  Take $f\in\left(
    M^{p_{0}}_{X_{0}},M^{p_{1}}_{X_{1}}
  \right)_{\kap,\infty}$.
  Let $(g,h)\in M^{p_{0}}_{X_{0}}\times M^{p_{1}}_{X_{1}}$ be an arbitrary decomposition $f=g+h$.
  Then the pair $\left(g_{R}\chi_{B(z,1)},h_{R}\chi_{B(z,1)}\right)$ becomes an element of $X_{0}\times X_{1}$
  and satisfies $f_{R}\chi_{B(z,1)}=g_{R}\chi_{B(z,1)}+h_{R}\chi_{B(z,1)}$.
  Thus, for all fixed $\lam>0$, we have
  \begin{align*}
    K\left(
      f_{R}\chi_{B(z,1)},\lam;X_{0},X_{1}
    \right)
    &\leq
    \norm{g_{R}\chi_{B(z,1)}\cd X_{0}}
    +\lam\norm{h_{R}\chi_{B(z,1)}\cd X_{1}}
    \\
    &\leq
    R^{-\frac{N}{p_{0}}}\norm{g\cd M^{p_{0}}_{X_{0}}}
    +\lam
    R^{-\frac{N}{p_{1}}}\norm{h\cd M^{p_{1}}_{X_{1}}}
    \\
    &=
    R^{-\frac{N}{p_{0}}}
    \left\{
        \norm{g\cd M^{p_{0}}_{X_{0}}}
        +\lam
        R^{\frac{N}{p_{0}}-\frac{N}{p_{1}}}
        \norm{h\cd M^{p_{1}}_{X_{1}}}
    \right\}.
  \end{align*}
  Setting $\Lam:=\lam R^{\frac{N}{p_{0}}-\frac{N}{p_{1}}}$,
  the arbitrariness of the decomposition implies
  \begin{align*}
    K\left(
      f_{R}\chi_{B(z,1)},\lam;X_{0},X_{1}
    \right)
    &\leq
    R^{-\frac{N}{p_{0}}}
    \Lam^{\kap}\Lam^{-\kap}
    K\left(
      f,\Lam;M^{p_{0}}_{X_{0}},M^{p_{1}}_{X_{1}}
    \right)
    \\
    &\leq
    R^{-\frac{N}{p_{0}}}
    \Lam^{\kap}
    \norm{f\cd\left(M^{p_{0}}_{X_{0}},M^{p_{1}}_{X_{1}}\right)_{\kap,\infty}}
  \end{align*}
  for all $\lam>0$.
  Therefore,
  in the light of $X\supset\left(X_{0},X_{1}\right)_{\kap,\infty}$,
  we obtain
  \begin{align*}
    R^{\frac{N}{p}}
    \norm{f_{R}\chi_{B(z,1)}\cd X}
    &\leq
    CR^{\frac{N}{p}}
    \norm{f_{R}\chi_{B(z,1)}
    \cd\left(X_{0},X_{1}\right)_{\kap,\infty}}
    \\
    &=CR^{\frac{N}{p}}
    \sup_{\lam>0}
    \lam^{-\kap}
    K\left(f_{R}\chi_{B(z,1)},\lam;X_{0},X_{1}\right)
    \\
    &\leq
    CR^{\frac{N}{p}}
    \sup_{\lam>0}
    \lam^{-\kap}
    R^{-\frac{N}{p_{0}}}
    \Lam^{\kap}
    \norm{f\cd\left(M^{p_{0}}_{X_{0}},M^{p_{1}}_{X_{1}}\right)_{\kap,\infty}}
    \\&=C
    \norm{f\cd\left(M^{p_{0}}_{X_{0}},M^{p_{1}}_{X_{1}}\right)_{\kap,\infty}}.
  \end{align*}
  Taking the supremum over $z\in\R^{N}$ and $R\in\left]0,1\right]$ leads the desired embedding.
\end{proof}
\begin{proof}[Proof of Proposition \ref{cor-LRI}]
Clearly, $q\leq p$.
Use Theorem \ref{LorentzRI} and Proposition \ref{XMoRI} with $X_{i}=L^{q_{i},r_{i}}$ ($i\in\left\{,0,1\right\}$).
\end{proof}
\subsection{Properties of BLM spaces}
We prove some properties of BLM spaces, with reference to \cite{KozonoYamazaki}.
\begin{prop}\label{prop:2-2}
  We have the following continuous embeddings
  \begin{align*}
    B^{0,1}_{p,q,\infty}
    \subset
    M^{p}_{q,\infty}
    \subset
    B^{0,\infty}_{p,q,\infty}
  \end{align*}
  for $1<q\leq p<\infty$.
\end{prop}
\begin{proof}
  Proposition \ref{prop:YO-LM} and $\norm{\cF^{-1}\phi_{j}\cd L^{1}}=\norm{\cF^{-1}\phi_{0}\cd L^{1}}$ (for $j\in\Z$) give that
  \begin{align*}
    \norm{\Del_{j}f\cd M^{p}_{q,\infty}}
    \leq
    \begin{cases}
      C\norm{\cF^{-1}\phi_{0}\mid L^{1}}
      \norm{f\cd M^{p}_{q,\infty}}
      &\text{for }j\geq 1
      \\
      C\norm{\cF^{-1}\phi_{(0)}\mid L^{1}}
      \norm{f\cd M^{p}_{q,\infty}}
      &\text{for }j=0
    \end{cases}
  \end{align*}
  holds for all $f\in M^{p}_{q,\infty}(\R^{N})$.
  We take the supremum over $j\in\N$ to get the second one.
  Next we fix $f\in B^{0,1}_{p,q,\infty}$.
  We can easily check that $f=\sum_{j\in\N}\Del_{j}f$ in $\cS'(\R^{N})$
  and can conduct the following calculation
  \begin{align*}
    \norm{\sum_{j=0}^{n}\Del_{j}f-\sum_{j=0}^{m}\Del_{j}f\cd M^{p}_{q,\infty}}
    &\leq
    \sum_{j=m+1}^{n}
    \norm{\Del_{j}f\cd M^{p}_{q,\infty}}
    \left(\leq\norm{f\mid B^{0,1}_{p,q,\infty}}<\infty\right)
    \\
    &\rightarrow 0
    \quad(\text{as }m,n\to\infty).
  \end{align*}
  It turns out that $f\in M^{p}_{q,\infty}$.
  The embedding is easily seen to be continuous.
\end{proof}
Note that we didn't use properties of Lorentz--Morrey spaces apart from their completeness and Proposition \ref{prop:YO-LM} (see also Proposition 2.11 in \cite{KozonoYamazaki}).
We next verify the embedding theorem of Sobolev-type.
\begin{prop}\label{prop:SE}
  {(Sobolev-type embedding)}
  Let $r,\kap\in\left[1,\infty\right]$, $1<q\leq p<\infty$ and $s\in\R$.
  We have the following continuous embeddings:
  \begin{enumerate}[{\rm(1)}]
    \item $B^{s,r}_{p,q,\kap}\subset B^{s-N/p,r}_{\infty}$.
    \item $B^{s,r}_{p,q,\kap}\subset B^{s-N(1-l)/p,r}_{p/l,q/l,\kap}$\quad for $l\in\left]0,1\right[$.
  \end{enumerate}
  Here, $B^{s-N/p,r}_{P}$ stands for inhomogeneous Besov spaces defined in the same way as inhomogeneous BLM spaces by replacing Lorentz--Morrey spaces with the standard $L^{P}$ space.
\end{prop}
To obtain this, we modify the proof of Theorem 2.5 in \cite{KozonoYamazaki}.
First we prepare the following estimate.
\begin{lem}\label{lem:SE}
    Let $1<q\leq p<\infty$ and $1\leq\kap\leq\infty$.
    There exists $C>0$ such that the inequalities
    \begin{align}\label{eq:2-2}
        \sup_{x\in\R^{N}}
        \left|\Del_{j}f(x)\right|
        \leq C2^{jN/p}
        \norm{\Del_{j}f\cd M^{p}_{q,\infty}}
        \leq C2^{jN/p}
        \norm{\Del_{j}f\cd M^{p}_{q,\kap}}
    \end{align}
    are valid
    for all $f\in M^{p}_{q,\infty}$ and for all $j\in\N$.
\end{lem}
We prove only the former inequality (the latter one is a consequence of the standard fact that $L^{q,\kap}\subset L^{q,\infty}$).
\begin{proof}[Proof of Lemma \ref{lem:SE}]
We write
$\Phi_{j}(\xi):=\phi_{j-1}(\xi)+\phi_{j}(\xi)+\phi_{j+1}(\xi)$ for $j\in\Z$ and $\Phi_{(0)}(\xi):=\phi_{(0)}(\xi)+\phi_{1}(\xi)$.
Then identities $\phi_{j}(\xi)\equiv\Phi_{j}(\xi)\phi_{j}(\xi)$, $\phi_{(0)}(\xi)\equiv\Phi_{(0)}(\xi)\phi_{(0)}(\xi)$
and $\Psi_{j}(x):=\cF^{-1}\Phi_{j}(x)=2^{Nj}\Psi_{0}(2^{j}x)$ hold.
Fix $x\in\R^{N}$.
We decompose $|\Del_{j}f(x)|$ as
\begin{align}\label{eq:SE-1}
  &\left|\Del_{j}f(x)\right|
  =
  \left|\cF^{-1}\left[\Phi_{j}\phi_{j}\cF f\right](x)\right|
  \leq
  \int_{\R^{N}}
  \left|\Psi_{j}(y)\right|
  \left|\Del_{j}f(x-y)\right|
  \,dy\notag
  \\
  &\leq
  2^{Nj}
  \left(\int_{B(0,2^{-j})}+
  \sum_{k=-j}^{\infty}\int_{B(0,2^{k+1})\setminus B(0,2^{k})}\right)
  \left|\Psi_{0}(2^{j}y)\right|
  \left|\Del_{j}f(x-y)\right|
  \,dy.
\end{align}
For the first term in \eqref{eq:SE-1}, we deduce that
\begin{align*}
  &
  \int_{B(0,2^{-j})}
  \left|\Psi_{0}(2^{j}y)\right|
  \left|\Del_{j}f(x-y)\right|
  \,dy
  \leq
  \int_{\R^{N}}
  \left|\Psi_{0}(2^{j}y)\right|
  \left|\Del_{j}f(x-y)\chi_{B(0,2^{-j})}(y)\right|
  \,dy
  \\
  \leq &\,
  \norm{\Psi_{0}(2^{j}y)\cd L^{\frac{q}{q-1},1}}
  \norm{\Del_{j}f\chi_{B(x,2^{-j})}\cd L^{q,\infty}}
  \leq
  C2^{\frac{Nj(1-p)}{p}}
  \norm{\Del_{j}f\cd M^{p}_{q,\infty}}.
\end{align*}
Note that there exists a positive integer $\Om=\Om_{N}$ such that
for all $r>0$,
\begin{align*}
  B(0,2r)\subset\bigcup_{m=1}^{\Om}B(z_{m},r)
  \quad\text{holds for some}\quad
  \left\{z_{m}\right\}_{m=1}^{\Om},
\end{align*}
and therefore we can pick $\left\{z_{m}\right\}_{m=1}^{\Om^{k}}$ satisfying $B(0,2^{k}r)\subset\bigcup_{m=1}^{\Om^{k}}B(z_{m},r)$ for each $k\in\N$ and for all $r>0$.
Since $\Psi_{0}=\cF^{-1}\Phi_{0}$ is a rapidly decreasing function,
it holds that
\begin{align*}
  & 2^{Nj}\int_{B(0,2^{k+1})\setminus B(0,2^{k})}
  \left|\Psi_{0}(2^{j}y)\right|
  \left|\Del_{j}f(x-y)\right|
  \,dy
  \\
  &\leq
  2^{Nj}
  \sum_{m=1}^{\Om^{k+1}}
  \int_{\R^{N}}
  \left|\Psi_{0}(2^{j}y)\chi_{\R^{N}\setminus B(0,2^{k})}(y)\right|
  \left|\Del_{j}f(x-y)\chi_{B(z_{m},1)}(y)\right|
  \,dy
  \\
  &\leq
  2^{jN/q}
  \sum_{m=1}^{2^{k\log_{2}\Om}\Om}
  \norm{\Psi_{0}(y)\chi_{\R^{N}\setminus B(0,2^{k+j})}(y)\cd L^{\frac{q}{q-1},1}}
  \norm{\Del_{j}f\chi_{B(x-z_{m},1)}\cd L^{q,\infty}}
  \\
  &\leq
  C2^{k\log_{2}{\Om}+jN/q}
  \norm{\Psi_{0}(y)\chi_{\R^{N}\setminus B(0,2^{k})}\cd L^{\frac{q}{q-1},1}}
  \norm{\Del_{j}f\cd M^{p}_{q,\infty}}
  \\
  &\leq
  C2^{jN/q}2^{-k}
  \norm{|y|^{1+\log_{2}{\Om}}
  \Psi_{0}(y)\chi_{\R^{N}\setminus B(0,2^{k})}
  \cd L^{\frac{q}{q-1},1}}
  \norm{\Del_{j}f\cd M^{p}_{q,\infty}}
  \\
  &\leq
  C2^{jN/q}2^{-k}
  \norm{\Del_{j}f\cd M^{p}_{q,\infty}}
\end{align*}
for each $k\in\N$.
Plugging these estimate into \eqref{eq:SE-1} yields \eqref{eq:2-2} for all $j\geq 1$ ($j=0$ case also follows by the same argument).
\end{proof}
\begin{proof}[Proof of Proposition \ref{prop:SE}]
  (1) It follows from \eqref{eq:2-2} immediately.

  (2) We see $\kap =\infty$ case (the rest case is proved in a similar way).
  Clearly
  \begin{align*}
    f^{\ast}(\lam)
    \leq
    \norm{f\mid L^{\infty}}
  \end{align*}
  holds for all $\lam>0$ and for all $f\in L^{\infty}$.
  Let $l\in\left]0,1\right[$ and $j\in\N$.
  Then, we have
  \begin{align*}
    \norm{\Del_{j}f\chi_{B(z,R)}\cd L^{q/l,\infty}}
    &
    \leq
    \sup_{\lam>0}
    \lam^{l/q}
    \left[\Del_{j}f\chi_{B(z,R)}\right]^{\ast}(\lam)^{l}
    \left[\Del_{j}f\chi_{B(z,R)}\right]^{\ast}(\lam)^{1-l}
    \\
    &=
    \left[
      \sup_{\lam>0}
      \lam^{1/q}
      \left[\Del_{j}f\chi_{B(z,R)}\right]^{\ast}(\lam)
    \right]^{l}
    \norm{\Del_{j}f\chi_{B(z,R)}\cd L^{\infty}}^{1-l}
    \\
    &\leq
    \norm{\Del_{j}f\chi_{B(z,R)}\cd L^{q,\infty}}^{l}
    \norm{\Del_{j}f\cd L^{\infty}}^{1-l}
    \\
    \eqref{eq:2-2}\quad
    &\leq
    CR^{\frac{Nl}{q}-\frac{Nl}{p}}
    2^{jN(1-l)/p}
    \norm{\Del_{j}f\cd M^{p}_{q,\infty}}
  \end{align*}
  for all $z\in\R^{N}$ and $R\in\left]0,1\right]$,
  and therefore the estimate
  \begin{align*}
    2^{\left(s-\frac{N(1-l)}{p}\right)j}
    \norm{\Del_{j}f\cd M^{p/l}_{q/l,\infty}}
    \leq
    C2^{sj}
    \norm{\Del_{j}f\cd M^{p}_{q,\infty}}
  \end{align*}
  holds for all $s\in\R$ and $j\in\N$.
\end{proof}

We use the reiteration theorem for Besov spaces in the proof of Proposition \ref{prop3-an}.
\begin{prop}\label{prop:Reint}
  {(reiteration theorem)}

  Fix $1\leq p\leq\infty$, $0<\kappa<1$ and $r, r_{0}, r_{1}\in\left[1,\infty\right]$.
  Let $s, s_{0}, s_{1}$ satisfy $s=(1-\kap)s_{0}+\kap s_{1}$.
  Then we have the formula
  \begin{align*}
    B^{s,r}_{p}
    =\left(B^{s_{0},r_{0}}_{p},B^{s_{1},r_{1}}_{p}\right)_{\kap,r}.
  \end{align*}
\end{prop}
\begin{proof}
  See Theorem 6.4.5.(1) in \cite{BLbook}.
\end{proof}

\subsection{Decay estimates of $S(t)$}
In this subsection, we collect the decay estimates of operators $S(t)$ on our function spaces.
We begin with BLM spaces.
\begin{prop}\label{decaySt}
  Let $1<q\leq p<\infty$, $r\in\left[1,\infty\right]$ and
  $s\geq\sig$.
  Then there exists $C>0$ such that the estimate
  \begin{align}\label{decaySt-1}
    \norm{S(t)f\cd B^{s,r}_{p,q,\infty}}
    \leq
    C\left(1+t^{\frac{-(s-\sig)}{\theta}}\right)
    \norm{f\cd B^{\sig,r}_{p,q,\infty}}
  \end{align}
  holds for all $t>0$ and for all $f\in B^{\sig,r}_{p,q,\infty}$.
  Furthermore, if $s>\sig$, the estimate
  \begin{align}\label{decaySt-2}
    \norm{S(t)f\cd B^{s,1}_{p,q,\infty}}
    \leq
    C\left(1+t^{\frac{-(s-\sig)}{\theta}}\right)
    \norm{f\cd B^{\sig,\infty}_{p,q,\infty}}
  \end{align}
  holds for all $t>0$ and for all $f\in B^{\sig,\infty}_{p,q,\infty}$.
\end{prop}
\begin{proof}
  First we prove \eqref{decaySt-1}.
  It suffices to show that the following inequality
  \begin{align*}
    \norm{\Del_{j}S(t)f\cd M^{p}_{q,\infty}}
    \leq C\left(1+t^{\frac{-(s-\sig)}{\theta}}\right)
    2^{-(s-\sig)j}
    \norm{\Del_{j}f\cd M^{p}_{q,\infty}}
  \end{align*}
  holds for all $j\in\N$.
  The case $j\geq 1$ can be treated in the same way as in the proof of Theorem 2.9 in \cite{KozonoYamazaki}, relying on Proposition \ref{prop:YO-LM} (see also \cite{BST, KS2022, Zhanpeisov23}).
  For $j=0$, the inequality follows by combining Proposition \ref{prop:YO-LM} with Lemma 2.1 in \cite{MYZ2008},
  as shown below:
  \begin{align*}
    \norm{\Del_{0}S(t)f\cd M^{p}_{q,\infty}}
    &=\norm{\cF^{-1}\left[\exp\left(-t|\xi|^{\theta}\right)\right]\ast\Del_{0}f
    \cd M^{p}_{q,\infty}}
    \\
    &\leq
    \norm{\cF^{-1}\left[\exp\left(-t|\xi|^{\theta}\right)\right]\cd L^{1}}
    \norm{\Del_{0}f\cd M^{p}_{q,\infty}}
    =C\norm{\Del_{0}f\cd M^{p}_{q,\infty}}.
  \end{align*}
  To obtain \eqref{decaySt-2}, we fix arbitrary $a>s$ and $m\in\N$,
  and perform the following calculation with \eqref{decaySt-1}:
  \begin{align*}
    \norm{S(t)f\cd B^{s,1}_{p,q,\infty}}
    &\leq
    \sum_{j=0}^{m-1}2^{(s-\sig)j}
    \norm{S(t)f\cd B^{\sig,\infty}_{p,q,\infty}}
    +\sum_{j=m}^{\infty}2^{-(a-s)j}
    \norm{S(t)f\cd B^{a,\infty}_{p,q,\infty}}
    \\
    &\leq
    C\left(2^{(s-\sig)m}
    +2^{-(a-s)m}\left(1+t^{-\frac{a-\sig}{\theta}}\right)
    \right)
    \norm{f\cd B^{\sig,\infty}_{p,q,\infty}}.
  \end{align*}
  Optimizing in $m$, we get the claim.
\end{proof}

We confirm the decay estimates of operators $S(t)$ on Lorentz--Morrey spaces.
\begin{prop}\label{decaySt-Lo}
  Let $1<q\leq p<\infty$ and $l\in\left]0,1\right]$.
  Then there exists $C>0$ such that the estimate
  \begin{align}\label{eq:2-6}
    \norm{S(t)f\cd M^{p/l}_{q/l,\infty}}
    \leq
    C\left(
      1+t^{-\frac{N(1-l)}{\theta p}}
    \right)
    \norm{f\cd M^{p}_{q,\infty}}.
  \end{align}
  holds for all $t>0$ and for all $f\in M^{p}_{q,\infty}$.
\end{prop}

\begin{proof}
  {\bf Case $l<1$.}
  By Propositions \ref{prop:2-2}, \ref{prop:SE} and \ref{decaySt},
  it is seen that
  \begin{align*}
    \norm{S(t)f\cd M^{p/l}_{q/l,\infty}}
    &\leq
    C\norm{S(t)f\cd B^{0,1}_{p/l,q/l,\infty}}
    \leq
    C\norm{S(t)f\cd B^{N(1-l)/p,1}_{p,q,\infty}}
    \\
    &
    \leq
    C\left(
      1+t^{-\frac{N(1-l)}{\theta p}}
    \right)
    \norm{f\cd B^{0,\infty}_{p,q,\infty}}
    \leq
    C\left(
      1+t^{-\frac{N(1-l)}{\theta p}}
    \right)
    \norm{f\cd M^{p}_{q,\infty}}
  \end{align*}
  holds for all $t>0$ and for all $f\in M^{p}_{q,\infty}$.

  \noindent
  {\bf Case $l=1$.}
  This is done in the proof of the estimate \eqref{decaySt-1} ($j=0$ case).
\end{proof}
\begin{rmk}
  One cannot drop ``$1+$'' in \eqref{eq:2-6};
  one arrives at a contradiction by considering $f(x)\equiv 1$ and taking the limit $t\to\infty$,
  for instance.
  In this article, however, the term ``$1+$'' is negligible,
  as we restrict our attention to local-in-time solvability and make use of the following inequality:
  \begin{align*}
    \left(1+t^{-\frac{N(1-l)}{\theta p}}\right)\chi_{\left]0,T\right[}(t)
    \leq Ct^{-\frac{N(1-l)}{\theta p}}
    \quad\text{for }0<T\leq 1\,\text{ and for }t>0.
  \end{align*}
\end{rmk}
\subsection{An estimate on the inhomogeneous term}
In this subsection,
we establish the following proposition that will be used in the proof of Theorem \ref{thm2};
this constitutes one of the novelties of this paper.
\begin{prop}\label{prop0}
  Let $T>0$ and $1<q\leq p<\infty$.
  There exists a positive constant $C$ depending on $T$ such that the inequality
  \begin{align*}
    \norm{\mu\cd B^{-\theta,\infty}_{p,q,\infty}}
    \leq C
    \norm{
      \int_{0}^{t}
      S(t-\tau)\mu\,d\tau
    \cd L^{\infty}(0,T;M^{p}_{q,\infty})}
  \end{align*}
  holds for all $\mu\in \left\{f\in\cS'(\R^{N})\cd\int_{0}^{t}
  S(t-\tau)f\,d\tau\in L^{\infty}(0,T;M^{p}_{q,\infty})\right\}$.
\end{prop}
To the best of the author's knowledge, this estimate seems not to have received attention, even for usual Besov spaces and for $\theta=2$.
We note that the proof below does not depend on the base space $M^{p}_{q,\infty}$ except for the convolution inequality of Young type (Proposition \ref{prop:YO-LM}).

\begin{proof}[Proof of Proposition \ref{prop0}]

  We set
  \begin{align*}
    C_{T}(\xi)=C_{T}^{\theta}(\xi):=
    \left(
      \int_{0}^{T}
      \frac{1-e^{-|\xi|^{\theta}t}}{|\xi|^{\theta}}
      \,dt
    \right)^{-1}
    =
    \frac{|\xi|^{2\theta}}{|\xi|^{\theta}T+e^{-|\xi|^{\theta}T}-1}
  \end{align*}
  and observe that
  \begin{align}\label{eq:0-2}
    C_{T/2^{\theta j}}(\xi)
    =
    2^{2\theta j}C_{T}(2^{-j}\xi)
  \end{align}
  holds for all $T>0$, $j\in\Z$
  and $\xi\in\R^{N}\setminus\{0\}$.
  We calculate as follows
  \begin{align}\label{eq:0-1}
    \cF\left[
      \int_{0}^{t}
      S(t-\tau)\mu\,d\tau
    \right](\xi)
    &=
    \int_{0}^{t}
    e^{-(t-\tau)|\xi|^{\theta}}\cF\mu(\xi)\,d\tau
    =
    \frac{1-e^{-|\xi|^{\theta}t}}{|\xi|^{\theta}}
    \cF\mu(\xi).
  \end{align}
  Then we obtain
  \begin{align*}
    \cF\Del_{j}\mu
    &=
    \phi_{j}(\xi)\cF\mu(\xi)\times 1
    =
    \phi_{j}(\xi)\cF\mu(\xi)
    C_{T}(\xi)
      \int_{0}^{T}
      \frac{1-e^{-|\xi|^{\theta}t}}{|\xi|^{\theta}}
      \,dt
    \\
    \eqref{eq:0-1}\,\,
    &=
    \Phi_{j}(\xi)
    C_{T}(\xi)
    \int_{0}^{T}
    \phi_{j}(\xi)
    \cF\left[
      \int_{0}^{t}
      S(t-\tau)\mu\,d\tau
    \right](\xi)
    \,dt
  \end{align*}
  for all $j\in\Z_{\geq 1}$.
  Here $\Phi_{j}$ is the same as in the proof of Lemma \ref{lem:SE}.
  Thanks to Proposition \ref{prop:YO-LM}, we see that
  \begin{align*}
    \norm{\Del_{j}\mu\cd M^{p}_{q,\infty}}
    &=
    \norm{
      \cF^{-1}
      \left[
        \Phi_{j}(\xi)C_{T}(\xi)
      \right]
      \ast
      \int_{0}^{T}
      \Del_{j}
      \left[
        \int_{0}^{t}
        S(t-\tau)\mu\,d\tau
      \right]
      \,dt
    \cd M^{p}_{q,\infty}}
    \\
    &\leq
    \norm{
      \cF^{-1}
      \left[
        \Phi_{j}(\xi)C_{T}(\xi)
      \right]
    \cd L^{1}}
    \int_{0}^{T}
    \norm{
      \Del_{j}
      \int_{0}^{t}
      S(t-\tau)\mu\,d\tau
    \cd M^{p}_{q,\infty}}
    \,dt
    \\
    &\leq
    \norm{
      \cF^{-1}
      \left[
        \Phi_{j}(\xi)C_{T}(\xi)
      \right]
    \cd L^{1}}
    \int_{0}^{T}
    \norm{
      \int_{0}^{t}
      S(t-\tau)\mu\,d\tau
    \cd B^{0,\infty}_{p,q,\infty}}
    \,dt
  \end{align*}
  holds for all $T>0$.
  Now we replace $T$ with $T/2^{\theta j}$ to get
  \begin{align}\label{eq:0-5}
    &\norm{\Del_{j}\mu\cd M^{p}_{q,\infty}}\notag
    \\
    &\leq
    \norm{
      \cF^{-1}
      \left[
        \Phi_{j}(\xi)C_{T/2^{\theta j}}(\xi)
      \right]
    \cd L^{1}}
    \int_{0}^{T/2^{\theta j}}
    \norm{
      \int_{0}^{t}
      S(t-\tau)\mu\,d\tau
    \cd B^{0,\infty}_{p,q,\infty}}
    \,dt.
  \end{align}
  Note that
  \begin{align*}
    C^{\sharp}_{T}
    :=T
    \int_{\R^{N}}\left|
    \frac{1}{(2\pi)^{N}}
    \int_{\text{supp}\Phi_{0}}
    \Phi_{0}(\xi)C_{T}(\xi)e^{\sqrt{-1}x\xi}
    \,d\xi
    \right|\,dx
    <+\infty
  \end{align*}
  because $\xi\mapsto\Phi_{0}(\xi)C_{T}(\xi)$ is a smooth function.
  Thus, \eqref{eq:0-2} gives
  \begin{align}\label{eq:0-4}
    \norm{
      \cF^{-1}
      \left[
        \Phi_{j}(\xi)C_{T/2^{\theta j}}(\xi)
      \right]
    \cd L^{1}}
    &=
    2^{2\theta j}
    \norm{
      \cF^{-1}
      \left[
        \Phi_{0}(2^{-j}\xi)C_{T}(2^{-j}\xi)
      \right]
    \cd L^{1}}
    \notag
    \\
    &=
    2^{2\theta j}
    \norm{
      \cF^{-1}
      \left[
        \Phi_{0}(\xi)C_{T}(\xi)
      \right]
    \cd L^{1}}
    =2^{2\theta j}T^{-1}C^{\sharp}_{T}.
  \end{align}
  Since the indices $j$ are non-negative,
  we deduce that
  \begin{align}\label{eq:0-3}
    \int_{0}^{\frac{T}{2^{\theta j}}}
    \norm{
      \int_{0}^{t}
      S(t-\tau)\mu\,d\tau
    \cd B^{0,\infty}_{p,q,\infty}}
    dt
    &\leq
    \frac{T}{2^{\theta j}}
    \sup_{0\leq t\leq\frac{T}{2^{\theta j}}}
    \norm{
      \int_{0}^{t}
      S(t-\tau)\mu\,d\tau
    \cd B^{0,\infty}_{p,q,\infty}}
    \notag
    \\
    \leq\,
    \frac{T}{2^{\theta j}}
    &
    \sup_{0\leq t\leq T}
    \norm{
      \int_{0}^{t}
      S(t-\tau)\mu\,d\tau
    \cd B^{0,\infty}_{p,q,\infty}}.
  \end{align}
  Plugging \eqref{eq:0-4} and \eqref{eq:0-3} into \eqref{eq:0-5} yields
  \begin{align*}
    2^{-\theta j}\norm{\Del_{j}\mu\cd M^{p}_{q,\infty}}
    &\leq
    C^{\sharp}_{T}
    \sup_{0\leq t\leq T}
    \norm{
      \int_{0}^{t}
      S(t-\tau)\mu\,d\tau
    \cd B^{0,\infty}_{p,q,\infty}}
    \notag
  \end{align*}
  for all $j\in\Z_{\geq 1}$.
  In a similar way, we can estimate $\norm{\Del_{0}\mu\cd M^{p}_{q,\infty}}$ as follows:
  \begin{align*}
    \norm{\Del_{0}\mu\cd M^{p}_{q,\infty}}
    \leq
    T
    \norm{
      \cF^{-1}
      \left[
        \Phi_{(0)}(\xi)C_{T}(\xi)
      \right]
    \cd L^{1}}
    \sup_{0\leq t\leq T}
    \norm{
      \int_{0}^{t}
      S(t-\tau)\mu\,d\tau
    \cd B^{0,\infty}_{p,q,\infty}}.
  \end{align*}
  From Appendix \ref{sec-A}, we have
  \begin{align}\label{eq:0-finite}
    T^{-1}C^{\flat}_{T}
    :=
    \norm{
      \cF^{-1}
      \left[
        \Phi_{(0)}(\xi)C_{T}(\xi)
      \right]
    \mid L^{1}}
    <\infty.
  \end{align}
  Therefore, taking the supremum of $2^{-\theta j}\norm{\Del_{j}\mu\cd M^{p}_{q,\infty}}$ over $j\in\N$ implies
  \begin{align*}
    \norm{\mu\cd B^{-\theta,\infty}_{p,q,\infty}}
    &\leq
    \left(C^{\flat}_{T}+C^{\sharp}_{T}\right)
    \sup_{0\leq t\leq T}
    \norm{
      \int_{0}^{t}
      S(t-\tau)\mu\,d\tau
    \cd B^{0,\infty}_{p,q,\infty}}.
  \end{align*}
  Finally we invoke Proposition \ref{prop:2-2} to reach the desired inequality.
\end{proof}

\begin{rmk}\label{Csharp-deponT}
  \quad
  \begin{enumerate}
    \item Because $C_{T}(\xi)$ has an exceptional point ($\xi=0$), the smoothness $\Phi_{(0)}(\xi)C_{T}(\xi)$ is not trivial. In fact, we can confirm that the one-dimensional function
    \begin{align*}
      \R\ni x\mapsto\Psi(x):=\frac{x^{2}}{x+e^{-x}-1}\in\R
    \end{align*}
    belongs to $C^{\infty}(\left[0,\infty\right[)$.
    Thus $\theta\in 2\Z_{\geq 1}$ case are okay, for the map $\xi\mapsto|\xi|^{2k}$ is smooth in $\R_{\xi}^{N}$ for all $k\in\Z_{\geq 1}$ and $C_{T}(\xi)=T^{-2}\Psi(T|\xi|^{2})$
    , but
    the rest cases require additional care (see Appendix \ref{sec-A}).
    \item
    Taking into account \eqref{eq:0-3},
    we cannot obtain homogeneous Besov-type space version of the inequality
    (see also Remark \ref{rmk:a-1}).
    This is consistent with Theorem \ref{thm1} (1) (we note that $B^{-\theta,1}_{p,q,\infty}\setminus\dot{B}^{-\theta,\infty}_{p,q,\infty}\neq\varnothing$).
    \item
    The coefficient $C^{\flat}_{T}+C^{\sharp}_{T}$ appearing in the proof of Proposition \ref{prop0} satisfies
    \begin{align*}
      (0<)\,
      C^{\flat}_{T}+C^{\sharp}_{T}
      \leq
      CT^{-1}\quad\text{for }T\sim 0.
    \end{align*}
  \end{enumerate}
\end{rmk}

\section{Proofs of main theorems}\label{sec3}
In this section, we show the existence of solution to \eqref{inteq} as a fixed point of the map defined below.
To begin with, we fix $T>0$ and $1<q\leq p<\infty$ and we write
\begin{align*}
  X_{T}:=
  \left\{
    u(x,t):\text{Lebesgue measurable in }\R^{N}\times\left]0,T\right[
    \cd
    \norm{u\mid X_{T}}<\infty
  \right\},
\end{align*}
where
\begin{align*}
  \norm{u\mid X_{T}}:=
  \sup_{0<t<T}\norm{u(\cdot,t)\mid M^{p}_{q,\infty}}.
\end{align*}
For $u\in X_{T}$ and $\mu\in\cS'(\R^{N})$,
we define operators $I$, $J$ and $\cK$ as follows:
\begin{align*}
  \cK[u]=
  \cK[u;\mu](x,t)
  &=
  I[\mu](x,t)
  +
  J[u](x,t)
  \\
  &:=
  \int_{0}^{t}
  S(t-\tau)\mu(x)
  \,d\tau
  +
  \int_{0}^{t}
  S(t-\tau)
  \left[|u|^{\gam-1}u(\cdot,\tau)\right](x)
  \,d\tau
  \\
  &=
  \int_{0}^{t}
  S(\tau)\mu(x)
  \,d\tau
  +
  \int_{0}^{t}
  S(\tau)
  \left[|u|^{\gam-1}u(\cdot,t-\tau)\right](x)
  \,d\tau.
\end{align*}
We take a constant $M>0$ (which is determined later by \eqref{determine:M} and \eqref{determine:M-2}) and set
\begin{align*}
  B_{M}:=
  \left\{
    u\in X_{T}\mid\norm{u\mid X_{T}}\leq M
  \right\}.
\end{align*}
Then $\left(B_{M},\norm{\cdot -\cdot\cd X_{T}}\right)$ is a complete metric space.
In subsequent subsections, it turns out that the map $\cK:B_{M}\to B_{M}$ is well-defined and is a contraction map.
\subsection{Estimates on $I[\mu]$}
In this subsection,
we show two estimates on the operator $I[\mu]$ which correspond to Theorem \ref{thm1} (\ref{thm1-2}) and \ref{thm1} (\ref{thm1-1}).
These kind of critical estimates were early obtained independently by Meyer \cite{MeyerWavelet} and Yamazaki \cite{Yamazaki2000} in Lorentz spaces.
Recently,
in homogeneous Besov spaces $\dot{B}^{s}_{p,\infty}$ \cite{COT2022},
in the predual of (nonlocal) Lorentz--Morrey spaces $\cM^{p}_{q,\infty}$ \cite{Fe2016} or in uniformly local weak Zygmund spaces \cite{IKT2024},
such critical estimate was obtained and played a key role.
\begin{prop}\label{propI-1}
  Fix $0<T<1$, $1<q\leq p<\infty$ and $$r\in\left]\max\left\{\frac{p}{q};\frac{Np}{N+\theta p}\right\},p\right[.$$
  Then
  $I[\mu]$ is a map from $B^{N/r-N/p-\theta,\infty}_{r,rq/p,\infty}$ to $X_{T}$, and its operator norm is finite.
\end{prop}
\begin{proof}
  We take $p_{0}\in\left[r,p\right[$ so that
  $q_{0}:=p_{0}q/p>1$ and $2p_{0}>p$ hold.
  Then we put
  $p_{1}:=pp_{0}/(2p_{0}-p)$, $q_{1}:=qp_{0}/(2p_{0}-p)$
  ($\Rightarrow p_{1}\geq q_{1}>1$).
  We let $\lam>0$, set
  \begin{align}\label{eq:3-1}
  	\eta=\eta(\lam):=
  	\lam^{a},\quad
  	a:=
  	\frac{\theta pp_{0}}{2N(p-p_{0})}
  \end{align}
  and decompose $I[\mu]$ as
  \begin{align*}
  	I[\mu](x,t)
  	&=
    \int_{0}^{t}
    S(\tau)\mu(x)\,d\tau
    =
  	\int_{0}^{\infty}
    \chi_{\left]0,t\right[}(\tau)
    S(\tau)\mu(x)
    \,d\tau
  	\\
  	&=
    \left(
      \int_{0}^{\eta}
      +
      \int_{\eta}^{\infty}
    \right)
    \chi_{\left]0,t\right[}(\tau)
    S(\tau)\mu(x)
    \,d\tau
  	=:
  	I_{\lam}^{0}[\mu](x,t)
  	+
  	I_{\lam}^{1}[\mu](x,t).
  \end{align*}
  By Proposition \ref{prop:2-2}, Proposition \ref{decaySt} and \eqref{eq:3-1}, we have
\begin{align*}
	&\norm{I_{\lam}^{0}[\mu](x,t)\cd M^{p_0}_{q_0,\infty}}
	\leq
	C\int_{0}^{\eta}
  \chi_{\left]0,t\right[}(\tau)
  \norm{S(\tau)\mu(x)\cd B^{0,1}_{p_0,q_0,\infty}}
  \,d\tau
	\\
  &\leq
	C\int_{0}^{\eta}
  \tau^{-1+\frac{N}{\theta p_0}-\frac{N}{\theta p}}
  \,d\tau
  \norm{\mu\cd B^{\frac{N}{p_0}-\frac{N}{p}-\theta,\infty}_{p_0,q_0,\infty}}
  =
	C\lam^{1/2}
	\norm{\mu\cd B^{\frac{N}{p_0}-\frac{N}{p}-\theta,\infty}_{p_0,q_0,\infty}}
\end{align*}
and
\begin{align*}
	&\norm{I_{\lam}^{1}[\mu](x,t)\cd M^{p_1}_{q_1,\infty}}
	\leq
	C\int_{\eta}^{\infty}
  \chi_{\left]0,t\right[}(\tau)
  \norm{S(\tau)\mu(x)\cd B^{0,1}_{p_1,q_1,\infty}}
  \,d\tau
	\\
  &\leq
	C\int_{\eta}^{\infty}
  \tau^{-1+\frac{N}{\theta p_1}-\frac{N}{\theta p}}
  \,d\tau
  \norm{\mu\cd B^{\frac{N}{p_1}-\frac{N}{p}-\theta}_{p_1,q_1,\infty}}
  =
	C\lam^{-1/2}
	\norm{\mu\cd B^{\frac{N}{p_1}-\frac{N}{p}-\theta}_{p_1,q_1,\infty}}.
\end{align*}
Now we use Proposition \ref{cor-LRI} in the form
$M^{p}_{q,\infty}
\supset\left(M^{p_0}_{q_0,\infty},M^{p_1}_{q_1,\infty}\right)_{1/2,\infty}$
to derive
\begin{align*}
	\norm{I[\mu](x,t)\cd M^{p}_{q,\infty}}
  &\leq
  C\sup_{\lam>0}
  \lam^{-1/2}
  \inf_{I[\mu]=I_{0}+I_{1}}
  \left\{
		\norm{I_{0}\cd M^{p_0}_{q_0,\infty}}
		+
		\lam\norm{I_{1}\cd M^{p_1}_{q_1,\infty}}
	\right\}
  \\
  &\leq
	C\sup_{\lam>0}
	\left\{
		\lam^{-1/2}
    \norm{I_{\lam}^{0}[\mu]\cd M^{p_0}_{q_0,\infty}}
		+
		\lam^{1/2}
    \norm{I_{\lam}^{1}[\mu]\cd M^{p_1}_{q_1,\infty}}
	\right\}
	\\
	&\leq
	C\left\{
    \norm{\mu\cd B^{\frac{N}{p_0}-\frac{N}{p}-\theta,\infty}_{p_0,q_0,\infty}}
    +
    \norm{\mu\cd B^{\frac{N}{p_1}-\frac{N}{p}-\theta,\infty}_{p_1,q_1,\infty}}
  \right\}
  \\
  &\leq
  C\norm{\mu\cd B^{\frac{N}{p_0}-\frac{N}{p}-\theta,\infty}_{p_0,q_0,\infty}}
  \leq
  C\norm{\mu\cd B^{\frac{N}{r}-\frac{N}{p}-\theta,\infty}_{r,\frac{rq}{p},\infty}}
\end{align*}
for all $t<T<1$.
In the last line, we used Proposition \ref{prop:SE}.
\end{proof}
\begin{prop}\label{prop3-an}
  Fix $0<T<1$.
  Let $1<q\leq p<\infty$.
  Then
  $I[\mu]$ is a map from $B^{-\theta,1}_{p,q,\infty}$ to $X_{T}$,
  and its operator norm is finite.
\end{prop}
\begin{proof}
  We set $\Psi_{j}(x):=
  \cF^{-1}\left[\Phi_{j}\right](x)$ for $j\in\Z_{\geq 1}$ and $\Psi_{0}(x):=
  \cF^{-1}\left[\Phi_{(0)}\right](x)$,
  where $\Phi_j$ and $\Phi_{(0)}$ are the same as in the proof of Lemma \ref{lem:SE}.
  Let $t\in\left]0,T\right[$.
  By using Proposition \ref{prop:YO-LM} and \ref{prop:2-2}, we have
  \begin{align}\label{eq:prop3-an-1}
    &\norm{\int_{0}^{t}S(t-\tau)\mu(x)\,d\tau\cd M^{p}_{q,\infty}}
    \leq C
    \sum_{j\in\N}
    \norm{\Del_{j}\int_{0}^{t}S(\tau)\mu(x)\,d\tau\cd M^{p}_{q,\infty}}
    \notag
    \\
    &=C
    \sum_{j\in\N}
    \norm{
    \int_{0}^{t}S(\tau)
    \left[\Psi_{j}\right]
    (x)\,d\tau
    \ast
    \Del_{j}\mu
    \cd M^{p}_{q,\infty}}
    \notag
    \\
    &\leq C
    \sum_{j\in\N}
    \norm{\int_{0}^{t}S(\tau)
    \left[\Psi_{j}\right]
    (x)\,d\tau
    \cd L^{1}}
    \norm{
    \Del_{j}\mu
    \cd M^{p}_{q,\infty}}.
  \end{align}
  Because $\mathrm{supp}\,\Phi_{j}\subset\mathrm{supp}\left[\sum_{l=-2}^{2}\phi_{j+l}\right]$ holds for $j\in\Z_{\geq 1}$,
  we have
  \begin{align*}
    &\norm{\int_{0}^{t}S(\tau)
    \left[\Psi_{j}\right]
    (x)\,d\tau
    \cd L^{1}}
    \leq
    \sum_{l=-2}^{2}
    2^{-\theta(j+l)}2^{\theta(j+l)}
    \norm{\Del_{j+l}
    \int_{0}^{t}
    S(\tau)\left[\Psi_{j}\right](x)
    \,d\tau
    \cd L^{1}}
    \\
    &\leq
    5\cdot2^{2\theta}
    2^{-\theta j}
    \max_{-2\leq l\leq 2}
    2^{\theta(j+l)}
    \norm{\Del_{j+l}
    \int_{0}^{t}
    S(\tau)\left[\Psi_{j}\right](x)
    \,d\tau
    \cd L^{1}}
    \\
    &\leq C
    2^{-\theta j}
    \norm{
    \int_{0}^{t}
    S(\tau)\left[\Psi_{j}\right](x)
    \,d\tau
    \cd B^{\theta,\infty}_{1}}
    \\
    &\leq C
    2^{-\theta j}
    \norm{
    \int_{0}^{\infty}
    \chi_{\left]0,t\right[}(\tau)
    S(\tau)\left[\Psi_{j}\right](x)
    \,d\tau
    \cd\left(B^{\theta-\theta r,\infty}_{1},
    B^{\theta+\theta r,\infty}_{1}\right)_{1/2,\infty}}.
  \end{align*}
  Here we wrote $\Del_{j}:=\Del_{0}$ and $\Psi_{j}:=\Psi_{0}$ for $j<0$,
  and we used Proposition \ref{prop:Reint} in the last
  for sufficiently small $r>0$.
  Because $\norm{
  \Psi_{j}
  \cd L^{1}}$ is independent of $j\in\N$,
  the decay estimate of $S(\tau)$ on the Besov spaces \cite[Theorem 3]{BST} reveals that
  \begin{align*}
    &\norm{\int_{0}^{t}S(\tau)
    \left[\Psi_{j}\right]
    (x)\,d\tau
    \cd L^{1}}
    \\
    &\leq C2^{-\theta j}
    \sup_{\lam>0}
    \lam^{-1/2}
    \bigg\{
      \int_{0}^{\lam^{1/2r}}
      \chi_{\left]0,t\right[}(\tau)
      \left(1+\tau^{r-1}\right)
      \norm{
      \Psi_{j}
      \cd B^{0,\infty}_{1}}
      \,d\tau
      \\
      &\qquad\qquad\qquad\qquad\quad
      +\lam
      \int_{\lam^{1/2r}}^{\infty}
      \chi_{\left]0,t\right[}(\tau)
      \left(1+\tau^{-1-r}\right)
      \norm{
      \Psi_{j}
      \cd B^{0,\infty}_{1}}
      \,d\tau
    \bigg\}
    \\
    &\leq C2^{-\theta j}
    \sup_{\lam>0}
    \lam^{-1/2}
    \bigg\{
      \int_{0}^{\lam^{1/2r}}
      \tau^{r-1}
      \,d\tau
      +\lam
      \int_{\lam^{1/2r}}^{\infty}
      \tau^{-1-r}
      \,d\tau
    \bigg\}
    \norm{
    \Psi_{j}
    \cd L^{1}}
    =
    C2^{-\theta j}
  \end{align*}
  holds for each $j\in\N$.
  Plugging this estimate in \eqref{eq:prop3-an-1} and taking the supremum over $t\in\left]0,T\right[$ yield the claim.
\end{proof}
\subsection{Contraction mapping argument}
In this section, we deal with $J[u]$ and complete the proof of main theorems.
\begin{lem}\label{lem3-1}
  Fix $0<T<1$.
  Assume that $\mu$ satisfies $I[\mu]\in B_{M/2}$ and $(p,q)$ satisfies
  $\gam<q\leq p$ and $p\geq{N(\gam-1)}/{\theta}$.
  Then the operator $u\mapsto\cK[u;\mu]$ is a map from $B_{M}$ to $B_{M}$.
\end{lem}
\begin{proof}
{\bf Case $p>{N(\gam-1)}/{\theta}$.}
  Take arbitrary $u\in B_{M}$.
  By utilizing Proposition \ref{decaySt-Lo} and Lemma \ref{lem:LM} (\ref{lem:LM-2}), we have
  \begin{align*}
    \norm{J[u](\cdot,t)\cd M^{p}_{q,\infty}}
    &\leq C
    \int_{0}^{t}
    \norm{S(t-\tau)\left[|u|^{\gam-1}u(\cdot,\tau)\right]\cd M^{p}_{q,\infty}}
    \,d\tau
    \\
    &\leq C
    \int_{0}^{t}
    \left\{1+(t-\tau)^{-\frac{N(\gam-1)}{\theta p}}\right\}
    \norm{|u|^{\gam-1}u(\cdot,\tau)\cd M^{p/\gam}_{q/\gam,\infty}}
    \,d\tau
    \\
    &\leq C
    T^{1-\frac{N(\gam-1)}{\theta p}}M^{\gam}
    \leq M/2
  \end{align*}
  for all $0<t<T<1$. This gives $\cK[u]\in B_{M}$.

  \noindent
  {\bf Case $p={N(\gam-1)}/{\theta}$.}
Fix any $u\in B_{M}$.
We take $r>0$ and $k(r)>0$ such that
\begin{align}\label{eq:3-2}
  &1<p/\gam<r<p=\frac{N(\gam-1)}{\theta}<k(r)
  \quad\text{and}\quad
  \frac{1}{p}
  =
  \frac{1/2}{r}+\frac{1/2}{k(r)}.
  \\
  &\left(\text{Then inequalities}\quad
  1<\frac{q}{\gam}<\frac{q}{p}r<q<\frac{q}{p}k(r)
  \quad\text{hold automatically.}
  \right)\notag
\end{align}
Then
we let $\lam>0$, set
\begin{align}\label{eq:3-3}
  \eta=\eta(\lam):=
  \lam^{b},\quad
  b:=
  \frac{\theta rk(r)}{N(k(r)-r)}
\end{align}
and decompose $J[u]$ as
\begin{align*}
  J[u](x,t)
  &=
  \int_{0}^{\infty}
  \chi_{\left]0,t\right[}(\tau)
  S(\tau)
  \left[|u|^{\gam-1}u(\cdot,t-\tau)\right](x)
  \,d\tau
  \\
  &=
  \left(
    \int_{0}^{\eta}
    +
    \int_{\eta}^{\infty}
  \right)
  \chi_{\left]0,t\right[}(\tau)
  S(\tau)
  \left[|u|^{\gam-1}u(\cdot,t-\tau)\right](x)
  \,d\tau
  \\
  &=:
  J_{\lam}^{0}[u](x,t)
  +
  J_{\lam}^{1}[u](x,t).
\end{align*}
By using Proposition \ref{decaySt-Lo}, \eqref{eq:3-2} and \eqref{eq:3-3}, we have
\begin{align*}
\norm{J_{\lam}^{0}[u]\cd M^{r}_{qr/p,\infty}}
&\leq
\int_{0}^{\eta}
\chi_{\left]0,t\right[}(\tau)
\norm{
S(\tau)
\left[|u|^{\gam-1}u(\cdot,t-\tau)\right](x)
\cd M^{r}_{qr/p,\infty}}
\,d\tau
\\
&\leq
C\int_{0}^{\eta}
\left(1+
\tau^{\frac{N}{\theta}\left(\frac{1}{r}-\frac{\gam}{p}\right)}
\right)
\chi_{\left]0,t\right[}(\tau)
\norm{
|u|^{\gam-1}u(\cdot,t-\tau)
\cd M^{p/\gam}_{q/\gam,\infty}}
d\tau
\\
&=
C\eta^{\frac{N}{\theta}\left(\frac{1}{r}-\frac{\gam}{p}\right)+1}
\sup_{0<t<T}\norm{u(\cdot,t)\cd M^{p}_{q,\infty}}^{\gam}
\leq
C\lam^{1/2}M^{\gam},
\end{align*}
and similarly we have
\begin{align*}
  \norm{J_{\lam}^{1}[u]\cd M^{k(r)}_{{qk(r)}/{p},\infty}}
  &\leq
  C\int_{\eta}^{\infty}
  \left(1+\tau\right)
  ^{\frac{N}{\theta}\left(\frac{1}{k(r)}-\frac{\gam}{p}\right)}
  \chi_{\left]0,t\right[}(\tau)
  \norm{
  |u|^{\gam}(\cdot,t-\tau)
  \cd M^{p/\gam}_{q/\gam,\infty}}
  \,d\tau
  \\
  &=
  C\eta^{\frac{N}{\theta}\left(\frac{1}{k(r)}-\frac{\gam}{p}\right)+1}
  \sup_{0<t<T}
  \norm{u(\cdot,t)\cd M^{p}_{q,\infty}}^{\gam}
  \leq
  C\lam^{-1/2}M^{\gam}.
\end{align*}
Proposition \ref{cor-LRI} in the form
$M^{p}_{q,\infty}\supset\left(M^{r}_{qr/p,\infty},M^{k(r)}_{qk(r)/p,\infty}\right)_{1/2,\infty}$
implies that
\begin{align}\label{determine:M}
	\norm{J[u](\cdot,t)\cd M^{p}_{q,\infty}}
  &\leq
  C\sup_{\lam>0}
  \lam^{-1/2}
  \inf_{J[\mu]=J_{0}+J_{1}}
  \left\{
		\norm{J_{0}\cd M^{r}_{qr/p,\infty}}
		+
		\lam\norm{J_{1}\cd M^{k(r)}_{qk(r)/p,\infty}}
	\right\}
  \notag
  \\
  &\leq
	C\sup_{\lam>0}
	\left\{
		\lam^{-1/2}
    \norm{J_{\lam}^{0}[u]\cd M^{r}_{qr/p,\infty}}
		+
		\lam^{1/2}
    \norm{J_{\lam}^{1}[u]\cd M^{k(r)}_{qk(r)/p,\infty}}
	\right\}
  \notag
	\\
	&\leq
	CM^{\gam}
  \leq M/2.
\end{align}
Thus we have
\begin{align*}
	\sup_{0<t<T}\norm{J[u](\cdot,t)\cd M^{p}_{q,\infty}}
	\leq\frac{M}{2}
\end{align*}
and therefore we find out $\cK[u]\in B_{M}$.
\end{proof}

\begin{lem}\label{lem3-2}
  Fix $0<T<1$.
  Assume that $\mu$ and $(p,q)$ satisfy the same conditions as in Lemma \ref{lem3-1}.
  Then the map
  $\cK:B_{M}\to B_{M}$, $u\mapsto\cK[u;\mu]$ is a contraction map.
\end{lem}
\begin{proof}
  We start by preparing the following estimate, in view of Lemma \ref{lem:LM}.
  For all $\rho\geq\sig>1$, $\nu>1$ and for all $f$, $g\in M^{\rho\nu}_{\sig\nu,\infty}$,
  we have
  \begin{align}\label{eq:mvt}
    &\norm{|f|^{\nu-1}f-|g|^{\nu-1}g\cd M^{\rho}_{\sig,\infty}}
    \notag
    \\
    &\qquad\qquad
    \leq
    \nu\norm{\max\{|f|;|g|\}^{\nu-1}|f-g|
    \cd M^{\rho}_{\sig,\infty}}\notag
    \\
    &\qquad\qquad
    \leq
    \nu\norm{\max\{|f|;|g|\}^{\nu-1}
    \cd M^{\frac{\nu\rho}{\nu-1}}_{\frac{\nu\sig}{\nu-1},\infty}}
    \norm{f-g\cd M^{\nu\rho}_{\nu\sig,\infty}}\notag
    \\
    &\qquad\qquad
    \leq
    \nu\left(
    \norm{f\cd M^{\nu\rho}_{\nu\sig,\infty}}^{\nu-1}
    +\norm{g\cd M^{\nu\rho}_{\nu\sig,\infty}}^{\nu-1}
    \right)
    \norm{f-g\cd M^{\nu\rho}_{\nu\sig,\infty}}.
  \end{align}

  \noindent
  {\bf Case $p=N(\gam-1)/\theta$.}
  Take arbitrary $u,v\in B_{M}$.
  Again we take $r$ and $k(r)$ as \eqref{eq:3-2} holds.
  Using
  Proposition \ref{cor-LRI} in the form
  $M^{p}_{q,\infty}\supset\left(M^{r}_{qr/p,\infty},M^{k(r)}_{qk(r)/p,\infty}\right)_{1/2,\infty}$,
  we observe that the estimate
  \begin{align*}
    &\norm{\cK[u]-\cK[v]\cd M^{p}_{q,\infty}}
    \\
    &=
    C\norm{
        \left(\int_{0}^{\eta}+\int_{\eta}^{\infty}\right)
        \chi_{\left]0,t\right[}(\tau)
        S(\tau)\left[|u|^{\gam-1}u(\cdot,t-\tau)-|v|^{\gam}v(\cdot,t-\tau)\right]\,d\tau
    \cd M^{p}_{q,\infty}}
    \\
    &\leq C\sup_{\lam>0}\bigg\{
    \lam^{-1/2}
    \int_{0}^{\eta}
    \chi_{\left]0,t\right[}(\tau)
    \norm{
        S(\tau)\left[|u|^{\gam-1}u(\cdot,t-\tau)-|v|^{\gam}v(\cdot,t-\tau)\right]
    \cd M^{r}_{qr/p,\infty}}
    d\tau
    \\
    &\qquad+\lam^{1/2}
    \int_{\eta}^{\infty}
    \chi_{\left]0,t\right[}(\tau)
    \norm{
        S(\tau)\left[|u|^{\gam-1}u(\cdot,t-\tau)-|v|^{\gam}v(\cdot,t-\tau)\right]
    \cd M^{k(r)}_{qk(r)/p,\infty}}
    d\tau\bigg\}
    \\
    &\leq C
    \sup_{\lam>0}
    \left(
        \lam^{-1/2}
        \int_{0}^{\eta}
        \tau^{\frac{N}{\theta}\left(\frac{1}{r}-\frac{\gam}{p}\right)}
        \,d\tau
        +\lam^{1/2}
        \int_{\eta}^{\infty}
        \tau^{\frac{N}{\theta}\left(\frac{1}{k(r)}-\frac{\gam}{p}\right)}
        \,d\tau
    \right)
    \\
    &\qquad\qquad\qquad\qquad\times
    \left(\norm{u\cd X_{T}}^{\gam-1}
    +\norm{v\cd X_{T}}^{\gam-1}\right)
    \norm{u-v\cd X_{T}}
    \\
    &\leq C
    \sup_{\lam>0}
    \left(
    \lam^{-1/2}\eta^{1+\frac{N}{\theta r}-\frac{N\gam}{\theta p}}
    +\lam^{1/2}\eta^{1+\frac{N}{\theta k(r)}-\frac{N\gam}{\theta p}}
    \right)
    M^{\gam-1}\norm{u-v\cd X_{T}}
  \end{align*}
  holds for any $\eta>0$,
  through the use of \eqref{eq:mvt} and Proposition \ref{decaySt-Lo}.
  At this stage, optimizing in $\eta$ and taking sufficiently small $M>0$,
  we derive
  \begin{align}\label{determine:M-2}
    \norm{\cK[u]-\cK[v]\cd M^{p}_{q,\infty}}
    \leq C\sup_{\lam>0}(1+1)M^{\gam-1}
    \norm{u-v\cd X_{T}}
    \leq
    \frac{1}{2}\norm{u-v\cd X_{T}}.
  \end{align}

  \noindent
  {\bf Case $p>N(\gam-1)/\theta$.}
  Utilizing \eqref{eq:mvt} and carrying out a similar procedure as in the proof of Lemma \ref{lem3-1}, the claim follows.
\end{proof}

\begin{pf1}
  From Lemma \ref{lem3-1}, Lemma \ref{lem3-2} and the contraction mapping theorem, we only have to show $\norm{I[\mu]\cd X_{T}}\leq M/2$.
  Assertion (\ref{thm1-1}) is proved by Proposition \ref{prop3-an}.
  We next prove (\ref{thm1-3}).
  This is established by following calculation.
  We use Proposition \ref{prop:2-2} and \ref{decaySt} to see that
  \begin{align*}
    \norm{I[\mu](\cdot,t)\cd M^{p}_{q,\infty}}
    &\leq C
    \int_{0}^{t}
    \norm{S(t-\tau)\mu\cd B^{0,1}_{p,q,\infty}}
    \,d\tau\notag
    \\
    &\leq C
    \int_{0}^{t}
    \left\{1+(t-\tau)^{s/\theta}\right\}
    \norm{\mu\cd B^{s,\infty}_{p,q,\infty}}
    \,d\tau\notag
    \\
    &\leq
    C\left(t+t^{1+s/\theta}\right)\norm{\mu\cd B^{s,\infty}_{p,q,\infty}}
    \leq
    C\left(T+T^{1+s/\theta}\right)\norm{\mu\cd B^{s,\infty}_{p,q,\infty}}
  \end{align*}
  holds for all $0<t\leq T$ and thus
  \begin{align}\label{eq:3-7}
    \norm{I[\mu]\cd X_{T}}
    &\leq
    C\left(T+T^{1+s/\theta}\right)
    \norm{\mu\cd B^{s,\infty}_{p,q,\infty}}
    \leq\frac{M}{2}
  \end{align}
  holds for sufficiently small $T>0$.
  Finally we prove (\ref{thm1-2}).
  Assume that $\mu\in\cS'(\R^{N})$ satisfies
  \begin{align*}
    \limsup_{j\to\infty}2^{(\ep-\theta)j}
    \norm{\Del_{j}\mu\cd
    M^{\frac{Np}{N+p\ep}}_{\frac{Nq}{N+p\ep},\infty}
    }<\del,
  \end{align*}
  where $\del>0$ is determined later.
  By definition of $\limsup$, there exists an integer $m\in\Z$ such that $2^{(\ep-\theta)j}\norm{\Del_{j}\mu\cd M^{\frac{Np}{N+p\ep}}_{\frac{Nq}{N+p\ep},\infty}}<\del$ holds for all $j\geq m$.
  Set $\mu_{1}:=\sum_{j=0}^{m}\Del_{j}\mu$ and $\mu_{2}:=\mu-\mu_{1}$.
  Then for every $j\in\N$, we have the following identities.
  \begin{align*}
    \Del_{j}\mu_{1}&=
    \begin{cases}
      \Del_{j}\mu
      &\text{for }j\leq m-1,
      \\
      \left(\Del_{m-1}+\Del_{m}\right)\Del_{j}\mu
      &\text{for }j=m, m+1,
      \\
      0&\text{for }j\geq m+2,
    \end{cases}
    \\
    \Del_{j}\mu_{2}&=
    \begin{cases}
      0&\text{for }j\leq m-1,
      \\
      \left(\Del_{m+1}+\Del_{m+2}\right)\Del_{j}\mu
      &\text{for }j=m, m+1,
      \\
      \Del_{j}\mu&\text{for }j\geq m+2.
    \end{cases}
  \end{align*}
  Thus we have $\norm{\mu_{2}\cd B^{\ep-\theta,\infty}_{\frac{Np}{N+p\ep},\frac{Nq}{N+p\ep},\infty}}<C\del$ and
  \begin{align*}
    \norm{\mu_{1}\cd B^{\sig,\infty}_{p,q,\infty}}
    &\leq C\max_{0\leq j\leq m+1}2^{\sig j}
    \norm{\Del_{j}\mu\cd M^{p}_{q,\infty}}
    \leq C
    \max_{0\leq j\leq m+1}2^{(\sig+\theta)j}2^{-\theta j}
    \norm{\Del_{j}\mu\cd M^{p}_{q,\infty}}
    \\
    &\leq C2^{(\sig+\theta)(m+1)}
    \norm{\mu\cd B^{-\theta,\infty}_{p,q,\infty}}
    \leq C2^{(\sig+\theta)(m+1)}
    \norm{\mu\cd B^{\ep-\theta,\infty}_{\frac{Np}{N+p\ep},\frac{Nq}{N+p\ep},\infty}}
  \end{align*}
  for some $\sig\in\left]-\theta,0\right[$.
  Now applying \eqref{eq:3-7} and Proposition \ref{propI-1} with $r=Np/(N+p\ep)$ to $\mu_{1}$ and $\mu_{2}$ respectively, we get
  \begin{align*}
    \norm{I[\mu]\cd X_{T}}
    &\leq
    \norm{I[\mu_{1}]\cd X_{T}}
    +
    \norm{I[\mu_{2}]\cd X_{T}}
    \\
    &\leq
    C\left(T+T^{1+\sig/\theta}\right)
    \norm{\mu_{1}\cd B^{\sig,\infty}_{p,q,\infty}}
    +
    C\norm{\mu_{2}\cd B^{\ep-\theta,\infty}_{\frac{Np}{N+p\ep},\frac{Nq}{N+p\ep},\infty}}
    \\
    &\leq
    C\left(T+T^{1+\sig/\theta}\right)
    2^{(\sig+\theta)m}
    \norm{\mu\cd B^{\ep-\theta,\infty}_{\frac{Np}{N+p\ep},\frac{Nq}{N+p\ep},\infty}}
    +
    C\del.
  \end{align*}
  Taking $\del>0$ and $T>0$ sufficiently small, we obtain $\norm{I[\mu]\cd X_{T}}\leq M/2$.
  \qed
\end{pf1}
\begin{rmk}\label{rmk3-5}
  Let $p$ and $q$ satisfy the same conditions in Theorem \ref{thm1}.
  If the inhomogeneous term $\mu$ satisfies one of the conditions in Theorem \ref{thm1},
  we find a number $s<0$ such that $\mu$ belongs to $B^{s,\infty}_{p,q,\infty}\subset B^{s-N/p,\infty}_{\infty}$.
  Then the solution $u(x,t)$ constructed above satisfy the initial condition in the following sense:
  $u(x,t)\to 0$ in the weak-$\ast$ topology of the space $B^{s-N/p,\infty}_{\infty}$ as $t\to+0$.
  Take arbitrary $\psi\in B^{N/p-s,1}_{1}$ and calculate as follows:
  \begin{align*}
    \left|\langle u(\cdot,t),\psi\rangle\right|
    &\leq
    \int_{0}^{t}
    \left|\langle
    S(t-\tau)\mu,\psi
    \rangle\right|
    \,d\tau
    +
    \int_{0}^{t}
    \left|\langle
    S(t-\tau)\left[|u|^{\gam-1}u(\cdot,\tau)\right],\psi
    \rangle\right|
    \,d\tau
    \\
    &\leq
    \int_{0}^{t}
    \norm{S(t-\tau)\mu\cd B^{s-N/p,\infty}_{\infty}}
    \norm{\psi\cd B^{N/p-s,1}_{1}}
    \,d\tau
    \\
    &\quad+
    \int_{0}^{t}
    \norm{S(t-\tau)\left[|u|^{\gam-1}u(\cdot,\tau)\right]
    \cd B^{s-N/p,\infty}_{\infty}}
    \norm{\psi\cd B^{N/p-s,1}_{1}}
    \,d\tau
    \\
    &\leq
    \norm{\psi\cd B^{N/p-s,1}_{1}}
    \bigg(
    \int_{0}^{t}
    \norm{S(t-\tau)\mu\cd B^{s,\infty}_{p,q,\infty}}
    \,d\tau
    \\
    &\qquad+C
    \int_{0}^{t}
    \norm{S(t-\tau)\left[|u|^{\gam-1}u(\cdot,\tau)\right]
    \cd B^{\max\left\{0,s+\frac{(\gam-1)N}{p}\right\},\infty}_{p/\gam,q/\gam,\infty}}
    \,d\tau
    \bigg)
    \\
    &\leq
    C\norm{\psi\cd B^{N/p-s,1}_{1}}
    \bigg(
    \int_{0}^{t}
    1\,d\tau
    \norm{\mu\cd B^{s,\infty}_{p,q,\infty}}
    \\
    &\qquad+
    \int_{0}^{t}
    (t-\tau)
    ^{\min\left\{0,-\frac{s}{\theta}-\frac{(\gam-1)N}{p\theta}\right\}}
    \norm{|u|^{\gam-1}u(\cdot,\tau)
    \cd B^{0,\infty}_{p/\gam,q/\gam,\infty}}
    \,d\tau
    \bigg)
    \\
    &\leq
    C\norm{\psi\cd B^{N/p-s,1}_{1}}
    \left(
    \norm{\mu\cd B^{s,\infty}_{p,q,\infty}}
    +
    \norm{u\cd X_{T}}^{\gam}
    \right)
    t^{\min\left\{1,1-\frac{s}{\theta}-\frac{(\gam-1)N}{p\theta}\right\}}.
  \end{align*}
  Conditions $p\geq{N(\gam-1)}/{\theta}$ and $s<0$ imply that the last term disappears in the limit $t\to 0$.
\end{rmk}

\begin{pf2}
  Let $p\geq q>\gam$ and $p\geq{N(\gam-1)}/{\theta}$.
  Assume that a function $u=u(x,t)\in L^{\infty}(0,T;M^{p}_{q,\infty}(\R^{N}))$ satisfies the integral equation \eqref{inteq} for almost all $(x,t)\in\R^{N}\times\left]0,T\right[$,
  i.e., we have
  \begin{align*}
    u(x,t)=I[\mu](x,t)+J[u](x,t)
  \end{align*}
  for almost all $(x,t)\in\R^{N}\times\left]0,T\right[$.
  The same calculation of the proof of Lemma \ref{lem3-1} (estimate on $J[u]$) leads
  $J[u](x,t)\in L^{\infty}(0,T;M^{p}_{q,\infty}(\R^{N}))$,
  hence the inclusion
  \begin{align*}
    I[\mu](x,t)=u(x,t)-J[u](x,t)&\in L^{\infty}(0,T;M^{p}_{q,\infty}(\R^{N})).
  \end{align*}
  Finally we apply Proposition \ref{prop0} to derive $\norm{\mu\cd B^{-\theta,\infty}_{p,q,\infty}}<\infty$.
  \qed
\end{pf2}

\appendix

\begin{ack}
  The author was supported in part by JST SPRING, Grant Number JPMJSP2108, and KAKENHI \#24KJ0623.
\end{ack}
\section{Estimate \eqref{eq:0-finite} for general $\theta>0$}\label{sec-A}
In this section, we confirm that
\begin{align}
  \norm{
    \cF^{-1}
    \left[
      \Phi_{(0)}(\xi)C_{T}(\xi)
    \right]
  \cd L^{1}}
  <\infty
\end{align}
is valid for all $\theta>0$,
where
\begin{align*}
  C_{T}(\xi)=C_{T}^{\theta}(\xi):=
  \left(
    \int_{0}^{T}
    \frac{1-e^{-|\xi|^{\theta}t}}{|\xi|^{\theta}}
    \,dt
  \right)^{-1}
  =
  \frac{|\xi|^{2\theta}}{|\xi|^{\theta}T+e^{-|\xi|^{\theta}T}-1}.
\end{align*}
We modify the argument of Lemma 2.1 in \cite{MYZ2008}.
\begin{proof}
  We first note that there exists $C>0$ such that
  \begin{align}\label{eq:A-1}
    \left|
    \cF^{-1}
    \left[
      \Phi_{(0)}(\xi)C_{T}(\xi)
    \right](x)
    \right|
    \leq C \quad\text{for all }x\in\R^{N}
  \end{align}
  because $\Phi_{(0)}(\xi)C_{T}(\xi)$ is integrable with respect to $\xi$.
  Thus we focus on $|x|\gg 1$.

  We define the operator
  \begin{align*}
    L(x,D):=
    \frac{xD_{\xi}}{\sqrt{-1}|x|^{2}}.
  \end{align*}
  Then we have
  \begin{align*}
    L(x,D)e^{\sqrt{-1}x\xi}
    =
    \frac{1}{\sqrt{-1}|x|^{2}}
    \sum_{k=1}^{N}x_{k}
    \left(
      \frac{\pl}{\pl\xi_{k}}
      e^{\sqrt{-1}x\xi}
    \right)
    =
    e^{\sqrt{-1}x\xi}.
  \end{align*}
  The conjugate operator is
  \begin{align*}
    L^{\ast}(x,D):=
    \frac{\sqrt{-1}xD_{\xi}}{|x|^{2}}.
  \end{align*}
  Thus we can write $\cF^{-1}
  \left[
    \Phi_{(0)}(\xi)C_{T}(\xi)
  \right]$ as
  \begin{align*}
    \cF^{-1}
    \left[
      \Phi_{(0)}(\xi)C_{T}(\xi)
    \right]
    =
    \frac{1}{(2\pi)^{N}}
    \int_{\R^{N}}
    e^{\sqrt{-1}x\xi}
    L^{\ast}(x,D)\left(
      \frac{|\xi|^{2\theta}\Phi_{(0)}(\xi)}{|\xi|^{\theta}T+e^{-|\xi|^{\theta}T}-1}
    \right)
    \,d\xi.
  \end{align*}
  Set a $C_{c}^{\infty}(\R^{N})$-function $\rho$ satisfying
  \begin{align*}
    \rho(\xi)=
    \begin{cases}
      1,\quad |\xi|\leq 1,
      \\
      0,\quad |\xi|>2
    \end{cases}
  \end{align*}
  and decompose $\cF^{-1}
  \left[
    \Phi_{(0)}(\xi)C_{T}(\xi)
  \right]$ as
  \begin{align*}
    \cF^{-1}
    \left[
      \Phi_{(0)}(\xi)C_{T}(\xi)
    \right]
    &=
    \frac{1}{(2\pi)^{N}}
    \int_{\R^{N}}
    e^{\sqrt{-1}x\xi}
    \rho\left(\frac{\xi}{\del}\right)
    L^{\ast}(x,D)
    \left(
      \Phi_{(0)}(\xi)C_{T}(\xi)
    \right)
    \,d\xi
    \\
    &\quad+
    \frac{1}{(2\pi)^{N}}
    \int_{\R^{N}}
    e^{\sqrt{-1}x\xi}
    \left[
      1-\rho\left(\frac{\xi}{\del}\right)
    \right]
    L^{\ast}(x,D)
    \left(
      \Phi_{(0)}(\xi)C_{T}(\xi)
    \right)
    d\xi
    \\
    &=:I+II,
  \end{align*}
  where $\del\in\left]0,1/2\right[$ to be fixed later.
  Recall that the function $\Psi(s):=\frac{s^{2}}{s+e^{-s}-1}$ belongs to $C^{\infty}(\left[0,\infty\right[)$.
  By mathematical induction on $k$, we can show that there exist constants $C^{k}_{j,i}$ which are independent of $T$ such that
  \begin{align*}
    \left(L^{\ast}\right)^{k}C_{T}(\xi)
    &=
    \left(L^{\ast}\right)^{k}
    \left(
      \frac{|\xi|^{2\theta}}{|\xi|^{\theta}T+e^{-|\xi|^{\theta}T}-1}
    \right)
    =
    T^{-2}
    \left(L^{\ast}\right)^{k}
    \left(
      \Psi(|\xi|^{\theta}T)
    \right)
    \\
    &=
    \frac{\sqrt{-1}^{k}}{|x|^{2k}}
    \sum_{j=1}^{k}
    T^{j-2}
    \Psi^{(j)}(|\xi|^{\theta}T)
    \sum_{\substack{0\leq i\leq k-j\\2i\leq k}}
    C^{k}_{j,i}|\xi|^{j\theta+2i-2k}(x\xi)^{k-2i}|x|^{2i}
  \end{align*}
  holds for all $k\geq 1$ and for all $\xi$.
  Note that
  $
    \left|\frac{d^{j}}{ds^{j}}\Psi\left(Ts\right)\right|
    \leq T^{j}C
  $
  for bounded $s$ and for $j\in\N$.
  Thus we have the estimate
  \begin{align*}
    \left|
    \left(L^{\ast}\right)^{k}
      C_{T}(\xi)
    \right|
    &\leq
    \frac{C_{k}}{|x|^{k}}
    \sum_{j=1}^{k}
    T^{j-2}
    |\xi|^{j\theta-k}
    \leq
    C_{k}T^{-1}|x|^{-k}|\xi|^{\theta-k}.
  \end{align*}
  For $I$, we have
  \begin{align*}
    |I|
    &\leq C\left|
    \int_{\R^{N}}
    e^{\sqrt{-1}x\xi}
    \rho\left(\frac{\xi}{\del}\right)
    L^{\ast}(x,D)
    \left(
      C_{T}(\xi)\Phi_{(0)}(\xi)
    \right)
    \,d\xi
    \right|
    \\
    &= C
    \int_{B_{2\del}(0)}
    \left|
    e^{\sqrt{-1}x\xi}
    \right|
    \left|
    L^{\ast}(x,D)
      C_{T}(\xi)
    \right|
    \,d\xi
    \\
    &\leq
    C\int_{B_{2\del}(0)}C_{1}T^{-1}|x|^{-1}|\xi|^{\theta-1}
    \,d\xi
    =
    CT^{-1}|x|^{-1}\del^{\theta-1+N}.
  \end{align*}
  Next we estimate $II$.
  Note that the estimate
  \begin{align*}
    \left|
    L^{\ast}(x,D)^{k}
    \left[
      1-\rho\left(\frac{\xi}{\del}\right)
    \right]
    \right|
    \leq
    \frac{C_{N,k}}{\del^{k}|x|^{k}}
  \end{align*}
  holds for $k\geq 1$.
  To estimate $II$, we integrate by parts $N$ times.
  \begin{align*}
    |II|
    &\leq
    C\int_{\R^{N}}
    \left|
    e^{\sqrt{-1}x\xi}
    \right|
    \left|
    \left(L^{\ast}\right)^{N}
    \left\{
    \left[
      1-\rho\left(\frac{\xi}{\del}\right)
    \right]
    L^{\ast}
    \left(
      C_{T}(\xi)\Phi_{(0)}(\xi)
    \right)
    \right\}
    \right|
    \,d\xi
    \\
    &\leq
    C\int_{\R^{N}\setminus B_{\del}(0)}
    \left[
      1-\rho\left(\frac{\xi}{\del}\right)
    \right]
    \left|
    \left(L^{\ast}\right)^{N+1}
    \left(
      C_{T}(\xi)\Phi_{(0)}(\xi)
    \right)
    \right|
    \,d\xi
    \\
    &\quad +
    C\int_{B_{2\del}(0)\setminus B_{\del}(0)}
    \sum_{k=1}^{N}
    \left|
      \left(L^{\ast}\right)^{k}
      \left[
        1-\rho\left(\frac{\xi}{\del}\right)
      \right]
      \left(L^{\ast}\right)^{N-k+1}
      \left(
        C_{T}(\xi)\Phi_{(0)}(\xi)
      \right)
    \right|
    \,d\xi
    \\
    &=:II_{1}+II_{2}.
  \end{align*}
  Because $B_{3}(0)\subset\left\{\xi\in\R^{N}\mid\Phi_{(0)}(\xi)\equiv 1\right\}$ holds, it is seen that
  \begin{align}\label{eq:a-3}
    |II_{1}|&\leq
    C\left(\int_{\text{supp}\Phi_{(0)}\setminus B_{3}(0)}
    +\int_{B_{3}(0)\setminus B_{\del}(0)}\right)
    \left|
    \left(L^{\ast}\right)^{N+1}
    \left(
      C_{T}(\xi)\Phi_{(0)}(\xi)
    \right)
    \right|
    \,d\xi
    \notag
    \\
    &\leq \frac{C_{N+1}}{|x|^{N+1}}
    \int_{\text{supp}\Phi_{(0)}\setminus B_{3}(0)}
    T^{-2}
    \max_{\xi}
    \left|D_{\xi}^{N+1}\left[
      \Psi(T|\xi|^{\theta})\Phi_{(0)}(\xi)
    \right]\right|\,d\xi
    \\
    &\quad+\int_{B_{3}(0)\setminus B_{\del}(0)}
    \left|
    \left(L^{\ast}\right)^{N+1}
      C_{T}(\xi)
    \right|
    \,d\xi
    \notag
    \\
    &\leq\frac{C}{T^{2}|x|^{N+1}}
    +
    \frac{C}{T|x|^{N+1}}
    \int_{B_{3}(0)\setminus B_{\del}(0)}
    |\xi|^{\theta-N-1}\,d\xi
    \leq
    \frac{C}{T^{2}|x|^{N+1}}+\frac{C\del^{\theta-1}}{T|x|^{N+1}}
    \notag
  \end{align}
  for $II_{1}$.
  For $II_{2}$, we have
  \begin{align*}
    |II_{2}|
    &\leq
    \sum_{k=1}^{N}
    \frac{C}{T\del^{k}|x|^{k}}
    \int_{B_{2\del}(0)\setminus B_{\del}(0)}
    |x|^{-(N-k+1)}|\xi|^{\theta-(N-k+1)}\,d\xi
    = CT^{-1}
    |x|^{-N-1}
    \del^{\theta-1}.
  \end{align*}
  Now we choose $\del=|x|^{-1}$.
  Considering \eqref{eq:A-1}, we obtain the following estimate
  \begin{align*}
    \left|
    \cF^{-1}
    \left[
      \Phi_{(0)}(\xi)C_{T}(\xi)
    \right](x)
    \right|
    \leq
    \frac{CT^{-1}}{\left(1+|x|\right)^{N+\theta}}
    +
    \frac{CT^{-2}}{\left(1+|x|\right)^{N+1}}
    \leq
    \frac{CT^{-2}}{\left(1+|x|\right)^{N+\min\{1,\theta\}}}.
  \end{align*}
  Thus we conclude that \eqref{eq:0-finite} holds for general $\theta>0$.
\end{proof}
\begin{rmk}[On \eqref{eq:a-3}]\label{rmk:a-1}
  Here, we used the estimate
  \begin{align*}
    \left|D_{\xi}^{N+1}\left[
      \Psi(T|\xi|^{\theta})\Phi_{(0)}(\xi)
    \right]\right|
    \leq
    \max_{\xi\in\mathrm{supp}\Phi_{(0)}\setminus B_{3}(0)}
    \left|D_{\xi}^{N+1}\left[
      \Psi(T|\xi|^{\theta})\Phi_{(0)}(\xi)
    \right]\right|
    \leq C_{N}<\infty.
  \end{align*}
  This attempt fails if we do with $\left\{\Psi(T|\xi|^{\theta})\Phi_{j}(\xi)\right\}_{j\in\Z}$.
\end{rmk}
\section{On the local Morrey-type spaces}\label{sec-B}
In this section, we make some comments on the local Morrey-type spaces.
\begin{prop}\label{prop:B-1}
  Let $1\leq q_{0}<q_{1}<p<\infty$.
  Then the continuous embeddings $M^{p}_{q_{0}}\supset M^{p}_{q_{1}}\supset L^{p,\infty}_{\mathrm{ul}}$
  hold and these inclusions are proper,
  i.e., there exists some functions $f$ such that $f\in M^{p}_{q_{0}}$ but $f\not\in M^{p}_{q_{1}}$.
  Here $L^{p,\infty}_{\mathrm{ul}}$ is defined as the set of measurable function $f$ on $\R^{N}$ satisfying
  \begin{align*}
    \norm{f\cd L^{p,\infty}_{\mathrm{ul}}}
    :=
    \sup_{z\in\R^{N}}
    \norm{f\chi_{B(z,1)}\cd L^{p,\infty}}
    <\infty.
  \end{align*}
\end{prop}
\begin{proof}
  The embedding $M^{p}_{q_{0}}\supset M^{p}_{q_{1}}$ comes from the H\"{o}lder inequality.
  We can prove the latter $M^{p}_{q_{1}}\supset L^{p,\infty}_{\mathrm{ul}}$ in the same manner as Lemma 1.7 in \cite{KozonoYamazaki}.
  Thus it is enough to give an example $f\in M^{p}_{q_{0}}\setminus M^{p}_{q_{1}}$.
  We modify the proof of Theorem 1.3 (i) in \cite{GHI2018}.
  For any $c>d>0$, we define the annulus $A^{c}_{d}:=B(0,c)\setminus B(0,d)$,
  and for each $k\in\Z_{\geq 1}$ and any $\del>0$, we define the following quantities:
  \begin{align*}
    &Q_{0}(k):=
    \int_{A^{1/k}_{1/(k+1)}}
    |x|^{-\frac{Nq_{0}}{p}}\,dx,
    \quad
    Q_{1}(k):=
    \int_{A^{1/k}_{1/(k+1)}}
    |x|^{-\frac{Nq_{1}}{p}-\del}\,dx,
    \\
    &a(k):=
    \max\left\{
      \left(\frac{Q_{1}(k)}{Q_{0}(k)}\right)^{\frac{1}{q_{1}-q_{0}}}
      ,\left(\frac{Q_{1}(k)}{\left|A^{1/k}_{1/(k+1)}\right|}\right)^{1/q_{1}}
    \right\}.
  \end{align*}
  The following two inequalities immediately follow from the definition of $a(k)$.
  \begin{align}\label{eq:B-1}
    \frac{Q_{1}(k)}{a(k)^{q_{1}}}
    \leq
    \frac{Q_{0}(k)}{a(k)^{q_{0}}}
    ,\quad
    \left(0<\right)\,
    \frac{Q_{1}(k)}{a(k)^{q_{1}}}
    \leq\left|A^{1/k}_{1/(k+1)}\right|.
  \end{align}
  Since the function $\ep\mapsto r_{k}(\ep):=\left|A^{1/(k+1)+\ep}_{1/(k+1)}\right|$
  is increasing and continuous, and satisfies $r_{k}(0)=0$ and $r_{k}\left(1/k-1/(k+1)\right)=\left|A^{1/k}_{1/(k+1)}\right|$,
  we can define the quantity $\ep(k)\in\left]0,1/k-1/(k+1)\right]$ by the formula
  $
    \left|A^{1/(k+1)+\ep(k)}_{1/(k+1)}\right|=\dfrac{Q_{1}(k)}{a(k)^{q_{1}}}
  $.
  Now we define a function $f:\R^{N}\to\R$ by
  \begin{align*}
    f(x):=
    \sum_{k=1}^{\infty}
    a(k)\chi_{A^{1/(k+1)+\ep(k)}_{1/(k+1)}}(x).
  \end{align*}
  We check that this $f$ satisfies $f\in M^{p}_{q_{0}}\setminus M^{p}_{q_{1}}$.
  Let $m=m(R)$ denote the unique integer which satisfies $1/(m+1)<R\leq 1/m$, for each $R\in\left]0,1\right]$.
  Then we have
  \begin{align*}
    \int_{B(0,R)}|f(x)|^{q_{0}}\,dx
    &=
    \sum_{k=m(R)}^{\infty}
    a(k)^{q_{0}}
    \left|A_{1/(k+1)}^{1/(k+1)+\ep(k)}\right|
    =
    \sum_{k=m(R)}^{\infty}
    a(k)^{q_{0}}
    \frac{Q_{1}(k)}{a(k)^{q_{1}}}
    \\
    \eqref{eq:B-1}\quad
    &\leq
    \sum_{k=m(R)}^{\infty}
    a(k)^{q_{0}}
    \frac{Q_{0}(k)}{a(k)^{q_{0}}}
    \leq
    \int_{B(0,2R)}
    |x|^{-\frac{Nq_{0}}{p}}\,dx
    =
    CR^{-\frac{Nq_{0}}{p}+N}
  \end{align*}
  and thus for arbitrary $R\in\left]0,1\right]$,
  the estimate
  \begin{align*}
    R^{\frac{N}{p}-\frac{N}{q_{0}}}
    \norm{f\chi_{B(0,R)}\cd L^{q_{0}}}
    \leq
    CR^{\frac{N}{p}-\frac{N}{q_{0}}}
    R^{-\frac{N}{p}+\frac{N}{q_{0}}}
    =C<+\infty
  \end{align*}
  holds.
  Similarly, for $\del\in\left]0,N\left(1-\frac{q_{1}}{p}\right)\right[$,
  we see that
  \begin{align*}
    \int_{B(0,R)}|f(x)|^{q_{1}}\,dx
    =
    \sum_{k=m(R)}^{\infty}
    a(k)^{q_{1}}
    \frac{Q_{1}(k)}{a(k)^{q_{1}}}
    \geq
    \int_{B(0,R)}
    |x|^{-\frac{Nq_{1}}{p}-\del}
    \,dx
    =CR^{-\frac{Nq_{1}}{p}-\del+N}.
  \end{align*}
  Therefore, sending $R\to 0$, we deduce
  \begin{align*}
    R^{\frac{N}{p}-\frac{N}{q_{1}}}
    \norm{f\chi_{B(0,R)}\cd L^{q_{1}}}
    \geq C
    R^{\frac{N}{p}-\frac{N}{q_{1}}}
    R^{-\frac{N}{p}-\frac{\del}{q_{1}}+\frac{N}{q_{1}}}
    =CR^{-\frac{\del}{q_{1}}}
    \rightarrow +\infty,
  \end{align*}
  which means $f\not\in M^{p}_{q_{1}}$.
\end{proof}
The function $f$ constructed above does not satisfy that $f(x)\leq a|x|^{-N/p}$ for $x\in B(0,\rho)$,
for any $a>0$ and $\rho>0$.
Finally, we remark that considering ``Morrey--Morrey'' spaces is not meaningful.
\begin{prop}\label{prop:B-2}
  Let $1\leq r\leq q\leq p<\infty$.
  Then we have $M^{p}_{r}
  =M^{p}_{M^{q}_{r}}$.
  Here,
  \begin{align*}
    M^{p}_{M^{q}_{r}}
    :=
    \left\{f\in L^{1}_{\mathrm{loc}}(\R^N)
        \cd\norm{f\cd M^{p}_{M^{q}_{r}}}
        :=\sup_{\substack{0<R\leq 1\\ z\in\R^{N}}}
        R^{\frac{N}{p}-\frac{N}{q}}
        \norm{f\chi_{B(z,R)}\cd M^{q}_{r}}
        <\infty
    \right\}.
  \end{align*}
\end{prop}
\begin{proof}
  Fix $R\in\left]0,1\right]$ and $z\in\R^{N}$.
  For all $f\in M^{p}_{M^{q}_{r}}$, we have the inequality
  \begin{align*}
    R^{\frac{N}{p}-\frac{N}{r}}
    \norm{f\chi_{B(z,R)}\cd L^{r}}
    &=
    R^{\frac{N}{p}-\frac{N}{q}}
    R^{\frac{N}{q}-\frac{N}{r}}
    \norm{f\chi_{B(z,R)}\chi_{B(z,R)}\cd L^{r}}
    \\
    &\leq
    R^{\frac{N}{p}-\frac{N}{q}}
    \norm{f\chi_{B(z,R)}\cd M^{q}_{r}}
    \leq
    \norm{f\cd M^{p}_{M^{q}_{r}}},
  \end{align*}
  and this implies $M^{p}_{r}\supset M^{p}_{M^{q}_{r}}$.
  Next we take any $f\in M^{p}_{r}$ and
  consider the function
  \begin{align*}
    Q:\left]0,1\right]^{2}\times\R^{N+N}
    &\rightarrow\R,
    \\
    \left(R,\tilde{R},z,w\right)
    &\mapsto
    Q\left[R,\tilde{R},z,w\right]
    :=R^{\frac{N}{p}-\frac{N}{q}}\tilde{R}^{\frac{N}{q}-\frac{N}{r}}
    \norm{f\chi_{B(z,R)}\chi_{B(w,\tilde{R})}\cd L^{r}}.
  \end{align*}
  Assume $R\geq\tilde{R}$.
  Then $p\geq q$ tells us $\left(R/\tilde{R}\right)^{\frac{N}{p}-\frac{N}{q}}\leq 1$
  and thus we get
  \begin{align*}
    Q\left[R,\tilde{R},z,w\right]
    \leq
    \left(\frac{R}{\tilde{R}}\right)^{\frac{N}{p}-\frac{N}{q}}
    \tilde{R}^{\frac{N}{p}-\frac{N}{r}}
    \norm{f\chi_{B(w,\tilde{R})}\cd L^{r}}
    \leq\norm{f\cd M^{p}_{r}}.
  \end{align*}
  Conversely assume $R\leq\tilde{R}$.
  Then $r\leq q$ tells us $\left(R/\tilde{R}\right)^{\frac{N}{r}-\frac{N}{q}}\leq 1$
  and thus we have
  \begin{align*}
    Q\left[R,\tilde{R},z,w\right]
    \leq
    \left(\frac{R}{\tilde{R}}\right)^{\frac{N}{r}-\frac{N}{q}}
    R^{\frac{N}{p}-\frac{N}{r}}
    \norm{f\chi_{B(z,R)}\cd L^{r}}
    \leq\norm{f\cd M^{p}_{r}}.
  \end{align*}
  Combining the above two yields $Q\leq\norm{f\cd M^{p}_{r}}$ and this finishes the proof.
\end{proof}

\end{document}